\documentclass[10pt]{article}
\usepackage[dvips]{color}
 \usepackage{amssymb}
 \usepackage{amsmath}
 \usepackage{amsfonts}
 \usepackage{mathrsfs}
\usepackage{amsmath, amsthm, amssymb}
 \newcommand{\be}{\begin{equation}}
 \newcommand{\ee}{\end{equation}}
 \newcommand{\bea}{\begin{eqnarray}}
 \newcommand{\eea}{\end{eqnarray}}
 \newcommand{\nn}{\nonumber}
 \newcommand{\rd}{\partial}
 \newcommand{\mc}{{\mathbb C}}

\newcommand{\mIs}{{{\mathbb I}}^{\,*}}
\newcommand{\mIm}{\widehat{I}_{M}}
\newcommand{\PP}{\boldsymbol{B}^{+}_1}

\newcommand{\Na}{{\mathbb N}}
\newcommand{\End}{\rm End }

\newcommand{\Deg}{\rm deg_{\widehat{\beta}}}

\newcommand{\Cof}{\rm Coef}

 \newtheorem{thm}{Theorem}
 \newtheorem{definition}{Definition}
 \newtheorem{cor}{Corollary}
 \newtheorem{lem}{Lemma}
 \newtheorem{pro}{Proposition}
 \newtheorem*{con}{Convention}
 \newtheorem{Rema}{Remark}
 \makeatletter
 \def\blfootnote{\xdef\@thefnmark{}\@footnotetext}
 \makeatother
 \usepackage[hidelinks]{hyperref}
 
\begin{document}

  \vspace{2cm}

 \begin{center}
    \font\titlerm=cmr10 scaled\magstep3
    \font\titlei=cmmi10 scaled\magstep4
    \font\titleis=cmmi7 scaled\magstep4
     \centerline{\titlerm  Study of the algebra of smooth integro-differential}
     \centerline{\titlerm  operators with applications}

    \vspace{1.5cm}

    \noindent{{\large
       A. Haghany$^{*}$\, and\, Adel Kassaian$^{\dagger}$
         }}
  \end{center}


 \begin{center}
    \font\titleis=cmmi7 scaled\magstep3
 \noindent{{\em Department of Mathematical Sciences, Isfahan University of Technology,\\ { 84156-83111} Isfahan, Iran}}\\
     \noindent{$^{*}$\em aghagh@cc.iut.ac.ir}\\
\noindent{$^{\dagger}$\em a.kassaian@gmail.com}
  \end{center}


 \vskip 2em

\begin{abstract}
We study the algebra of integro-differential operators with smooth
coefficients and kernels on a subspace of $C^{\infty}[a,b]$. We find
a normal form for elements of this algebra and determine its unit
group. The formulation of inverses gives explicit solutions of
inhomogeneous linear Volterra integro-differential equations and
Volterra integral equations of first kind with smooth kernels.\\ \\


{{\bf Mathematics Subject:}
   16W20;\, 12H05;\, 34A12;\, 45D05;\,45Jxx,\, 46A61;\, 47B38}

\end{abstract}

\vskip 4em
 \section{Introduction }\label{int}

The theory of differential operators and differential polynomials
starts primarily with a commutative algebra $\mathcal{F}$ over a
filed $K$ having a $K$-linear map $\rd:\mathcal{F}\to\mathcal{F}$
satisfying the rule $\rd(a\,b)=\rd(a)\,b+a\, \rd(b)$ for all $a$,
$b$ in $\mathcal{F}$. Then various subalgebras such as
$\mathcal{F}[\rd]$ of $\End_K\mathcal{F}$ are studied and used as a
tool to transfer analytic properties to algebraic ones and
conversely, with application in boundary value problems of
differential equations. This type of research dates back to
\cite{Kaplan} \cite{Ritt} \cite{Kolchin}. However, relatively more
recent research has focused on the notion of integro-differential
algebras arising as certain subalgebras of $\End_K\mathcal{F}$ by
adjoining $\int$ as a right inverse to $\rd$. Integro-differential
operators have been the subject of study from different angles and
many authors have contributed to their studies. In \cite{ij} and
\cite{bav}  algebraic aspects of integro differential operators on a
polynomial algebra has been thoroughly investigated; some other
relevant algebras such as Jacobian algebras and generalized Weyl
algebras are also considered. In the cited papers, for some of these
algebras one finds the ideal structure, characterization of simple
modules, exact values of dimensions (Krull dimension,
Gelfand-Kirillov dimension, weak and global homological dimensions)
and determination of group of automorphisms. See \cite{ros2} in
which integro-differential operators with polynomial coefficients
have been characterized in terms of skew polynomial constructions.
Another equally important motivation for research in algebras of
integro-differential operators is due to the useful role they play
in the general theory of differential equations. It should also be
mentioned that algebras of integro-differential operators have been
considered from both analytical and topological view points.
Recently in \cite{ros4} the authors have introduced a topology in
which not only the ring operations but also differentiation and
integration operators ($\rd$ and $\int$) are uniformly continuous
functions. Then by producing a {\em complete} integro differential
algebra they show that such an algebra enjoys $exp$ and $log$
mappings, hence every monic differential equation possesses a
regular fundamental system of solutions.\\ \\A classical example for
$\mathcal{F}$ is the $\mc$-algebra of complex valued infinitely
differentiable functions $C^{\infty}[a,b]$. There are operators
$H:=x,\hspace{0.1cm} \rd:=\frac{d}{d x}$ and $\hspace{0.1cm}
\int:=\int^{x}_{a}dz$ acting on functions $\phi(x)\in
C^{\infty}[a,b]$ by: $x\phi(x)$, $\frac{d}{d x}\phi(x)$ and
$\int^{x}_{a} \phi(z) dz$ respectively. A natural question is: what
subspaces $W\subset C^{\infty}[a,b]$ are maximal with respect to the
property that $\rd\int\phi=\phi=\int\rd\phi$ \,? If on $\phi(x)\in
C^{\infty}[a,b]$ we have $\rd\int=1=\int\rd$ then $\phi(a)=0$ and
since $\int \rd^{\,n+1}=\rd^{\,n}$ one deduces
$\frac{d^n}{dx^n}\phi_(x)|_{x=a}=0, \hspace{0.2cm} n\ge1$. Such
functions $\phi(x)$ form a unique largest subspace denoted by
$M[a,b]$ on which $\rd$ and $\int$ are two sided inverses. A well
known example of a member of $M[a,b]$ is
the function, \be\phi(x)=\begin{cases}\text{e}^{\,\,1\over \Large{a-x}} & \mbox { if }\,\,\, a < x \le b, \\
0  & \mbox { if }\,\,\,  x = a.
\end{cases}\nn\ee It is well known (e.g. \cite{rudi}) that $C^{\infty}[a,b]$ is a Frechet
space with seminorms $\displaystyle
\|\phi\|_{m}=\sum_{k=0}^{m}\sup_{x \in \,[a,b]}|\rd^k_{x}\phi(x)|$,
$m\in{\mathbb Z}_{+}$ and the metric, \be
d(\phi,\psi)=\sum_{k=0}^{\infty} 2^{-k}{\|\phi-\psi\|_{k}\over
{1+\|\phi-\psi\|_{k}}}.\nn\ee The same applies to $M[a,b]$ since it
is a closed subspace of $C^{\infty}[a,b]$. We note that $M[a,b]$ is
closed under action of Volterra operator $\mathcal{V}$ given by: \be
\mathcal{V}(\phi)(x)=\int^{x}_{a} dz K(x,z) \phi(z),
\hspace{0.5cm}\phi\in M[a,b],\hspace{0.2cm} x\in[a,b], \nn\ee where
the so called Volterra kernel $K(x,y)$ is a smooth function on
$\triangle=\{(x,y)\,|\, a \leq y \leq x\leq b \}$). The space
$M[a,b]$ is also closed under multiplication by a function in
$C^{\infty}[a,b]$. We are interested in the algebra of linear
operators (endomorphisms of $M[a,b]$) generated by  $\rd_{x}$ and
the sets $\{ P(x) |\,P(x)\in C^{\infty}[a,b]\}$ and
$\{\,\int^{x}_{a} f(x,z) dz |\,f(x,y)\in C^{\infty}[\triangle]\}$.
The actions on $\phi(x)\in M[a,b]$ are given by: \be P(x):\phi(x)
\mapsto P(x) \phi(x),\hspace{0.6cm}
 \int^{x}_{a} f(x,z) dz:\phi(x) \mapsto \int^{x}_{a} dz f(x,z) \phi(z)
,\hspace{0.6cm} \rd_{x}:\phi(x) \mapsto
\rd_{x}\phi(x).\label{fati}\ee We call this {\em the algebra of
smooth integro-differential operators} and denote it by $\mIm$. Note
that a linear map $T:M[a,b]\to M[a,b]$ is continuous if and only if
for each $k \in \mathbb{Z}_{+}$ there exist $C>0$ and $m \in
\mathbb{Z}_{+}$ such that for all $\phi\in M[a,b],$\be
\|T(\phi)\|_{\,k}\leq C \| \phi\|_{m}. \nn\ee It is routine to
verify that all the operators in (\ref{fati}) are continuous. The
purpose of this paper is the study of algebra $\mIm$. We shall find
a unique expression for each element of $\mIm$ (Theorem \ref{emam})
and by a detailed analysis of invertibility through projection
operators and a degree function in $\mIs$, an algebra isomorphic to
$\mIm$ (Theorem \ref{ershi}, Definition \ref{khalil} and Proposition
\ref{salar}), we finally determine the unit group $\mathcal{U}$ of
$\mIs$ (Theorem \ref{yarali}).\\\\ The explicit formulation of
inverses allows us to tackle some types of linear integral and
integro-differential equations in Section \ref{app}. The unit group
is partitioned as $\mathcal{U}=W_{1}\cup W_{2}\cup W_{3}$, such that
operators corresponding to elements of $W_1$ can be extended
(Proposition \ref{Haghany}) to operators $C[a,b]\to C[a,b]$. The
application is solution of Volterra integral equation of second kind
(known from classical analysis). Corresponding operators of $W_2$
have extensions $D^{n}[a,b] \to C[a,b]$, where ${D}^{n}[a,b]$ is the
space of $n$ times continuously differentiable functions
$\phi(x):[a,b]\to\mc$ such that ${\rd^{k}_{x}}\phi(x)|_{x=a}=0$ for
$k=0\cdots,n-1$. Application is causal solution of
integro-differential equation (\ref{ghar}) of degree $n$ (also known
from classical analysis). Finally corresponding operators of $W_3$
have extensions $C[a,b]\to D^{n}[a,b]$ by which we obtain a new
result for solution of Volterra integral equation of first kind with
smooth kernel $K(x,y)$ of degree $n$ (see \ref{nsmooth}), that is,
for \be \int^{x}_{a} dz K(x,z) y(z)=g(x), \nn\ee we find the
solution as, \be y(x)=\mathcal{O}(x, \rd_{x})g(x)+\int^{x}_{a} dz
R(x,z)\, g(z),\nn\ee where $\mathcal{O}(x, \rd_{x})$ is determined
explicitly
while $R(x,y)$ is found by a series of integrals.\\

Finally to illustrate the role of function space $M[a,b]$ it should
be mentioned that the spaces $C[a,b]$ and $D^{n}[a,b]$ (for $n\ge1$)
are Banach spaces and the largest subspace of $C[a,b]$ on which all
these invertible operators appear as automorphisms is in fact
Frechet space $M[a,b]$. This is because \be
M[a,b]=\bigcap_{n=1}^{\infty} D^{n}[a,b].\nn\ee


 \section {Algebra of smooth integro-differential operators on $M[a,b]$}
 \label{melika}

For the rest of this paper we adopt the following notations and definitions.\\

 $\mathbb{N} :=\{1, 2, \ldots \}$ is the set of natural numbers.
 $\mathbb{Z}_{+} :=\{0,1, 2, \ldots \}$ is the set of non-negative
 integers.
 $\mc$ is the
field of complex numbers. $a$ and $b$ are fixed real parameters,
$a<b$. $C[a,b]$ is the space of complex valued continuous functions
on $[a,b]$. $C^{\infty}[a,b]$ is the space of smooth functions on
$[a,b]$. \,$\boldsymbol{B_1}$ is the algebra of smooth functions
$P(x):[a,b] \to \mc$ with point-wise addition and multiplication of
functions. $\triangle$ is the subset of $[a,b] \times [a,b]$ defined
by $\triangle=\{(x,y)\,|\, a \leq y \leq x\leq b \}$.
$C^{\,n}[\Delta]$ is the space of functions $f(x,y):\triangle \to
\mc$, for which all partial derivatives
$\rd^{\,i}_{x}\rd^{\,j}_{y}f(x,y)$ and
$\rd^{\,i}_{y}\rd^{\,j}_{x}f(x,y)$, where $i+j=k$ ($k=0,1,\cdots,
n$), exist and are continuous. $\boldsymbol {B_2}$ is the algebra of
smooth functions $f(x,y):\triangle \to \mc$ with point-wise addition
and multiplication of functions. Vector space $M[a,b]$ is the space
of all smooth functions $\phi(x):[a,b]\to\mc$ for which $\phi(a)=0$
and $\forall n\in \mathbb{N}$, $\frac{d^n}{dx^n}\phi_(x)|_{x=a}=0$.
  $\End_\mc(M)$ is the algebra of all $\mc$-linear
maps from $M[a,b]$ to $M[a,b]$. The algebra $\mIm$ is the subalgebra
of $\End_\mc(M[a,b])$, generated by operators $\{ P(x) |\,P(x)\in
{\boldsymbol{B_1}}\}$, $\rd_{x}$ and $\{ \,\int^{x}_{a} f(x,z) dz
|\,f(x,y)\in \boldsymbol{B_2}\}$, acting on $\phi(x)\in M[a,b]$ as
 $P(x):\phi(x) \mapsto P(x) \phi(x)$, $ \rd_{x}:\phi(x) \mapsto \rd_{x}\phi(x)$,
  $\int^{x}_{a} f(x,z) dz:\phi(x) \mapsto \int^{x}_{a} f(x,z) \phi(z) dz$
  respectively. $\theta(x-y)$ is the Heaviside step function, $\theta(x-y)=1$ for $x\geq y$ and  $\theta(x-y)=0$ for $x<
  y$. For $P(x)\in {\boldsymbol{B_1}}$, by $P(x)\not\equiv 0$ we mean $P(x)$ is not
zero
 every where on $[a,b]$.\\\\

We begin our study of algebra $\mIm$ by noting that the operator
 $\int^{x}_{a} dz \,f(x,z)\in \mIm$, acts on $\phi(x)\in M[a,b]$ as;
\be \int^{x}_{a} dz\,f(x,z)\phi(z)= \int^{b}_{a} dz\,
f(x,z)\theta(x-z)\phi(z).\nn\ee Thus Volterra operator $\int^{x}_{a}
dz f(x,z)\in \mIm$ is equivalent to {\em Fredholm } operator
$\int^{b}_{a} dz\,f(x,z)\theta(x-z)$. From now on for $f(x,y)\in
\boldsymbol{B_2}$ we denote the operator $\int^{x}_{a} dz f(x,z)$ by
$f(x,y)\theta(x-y)$.

\begin{lem}\label{lemf4}{The set of operators \be\mathcal{G}=\{ g(x,y) \theta(x-y)\, | \,g(x,y)\in
\boldsymbol{B_2}\}\subset \mIm\nn\ee is closed under multiplication
of operators (on $M[a,b]$) such that for
$\,g_{1}(x,y),\,g_{2}(x,y)\in {\boldsymbol{B_2}}$ we have,
$\forall\, \phi(x)\in M[a,b]$, \be\int^{x}_{a}\, dt \,
g_2(x,t)(\int^{\,t}_{a} dz\, g_1(t,z)\phi (z)) \nn = \int^{x}_{a}
dz\, g_3(x,z) \phi (z),\nn\ee where \be g_{3}(x, y)=\int^{\,x}_{y}
dz\,g_2(x,z) g_1(z,y)\in {\boldsymbol {B_2}}. \nn\ee}
\end{lem}

{\bf Proof.} {By acting operator $g_1(x,y)\theta(x-y)$ and then
$g_2(x,y)\theta(x-y)$ in row on $\phi(x)\in M[a,b]$, \bea
&&\int^{x}_{a}\, dt \, g_2(x,t)(\int^{\,t}_{a} dz\, g_1(t,z)\phi
(z))={\int^{b}_{a}} dt
\,g_2(x,t)\,\,\theta(x-t)\,\,\big({\int^{b}_{a}}\,dz\,g_1(t,z)\,\theta(t-z)\phi(z)\big)
\nn\\&&={\int^{b}_{a}} dt\, \big({\int^{b}_{a}}\, d z\,\phi(z)\,
g_2(x,t)\,\theta(x-t)
\,g_1(t,z)\,\theta(t-z)\,\big)\nn\\&&={\int^{b}_{a}} dz\,
\big({\int^{b}_{a}}\, d t\,\phi(z)\, g_2(x,t)\,\theta(x-t)
\,g_1(t,z)\,\theta(t-z)\,\big)\nn\\&&={\int^{b}_{a}} dz\,
\phi(z)\big({\int^{b}_{a}}\, d t\, g_2(x,t)\,\theta(x-t)
\,g_1(t,z)\,\theta(t-z)\,\big)={\int^{b}_{a}}
dz\,\phi(z)\,\big(\,\theta(x-z)\,{\int^{x}_{z}}\, d t\,
g_2(x,t)g_1(t,z)\,\big)\nn\\&&={\int^{x}_{a}}
dz\,(\,\,{\int^{x}_{z}}\, d t\,
g_2(x,t)g_1(t,z)\,\big)\phi(z)\,.\nn\eea} In the third line the
order of integration is changed using the facts that the function
\be h(z,t)=\phi(z)\, g_2(x,t)\,\theta(x-t)
\,g_1(t,z)\,\theta(t-z),\nn\ee defined on $\Omega=[a,b]\times[a,b]$
is bounded and piecewise continuous. \qed

\begin{con}\label{char}{Throughout the paper we denote the
composition of two operators by $*$ and the action of operators on
function space by dot ($\cdot$). Therefore for general operators
${\bf{i}}, {\bf{j}}$ we have, \be({\bf{i}}*{\bf{j}})\cdot
\phi(x)=({\bf{i}}\cdot({\bf{j}}\cdot\phi(x)))\,\,.
\nn\label{four}\ee Moreover in $\mIm$ we denote powers by underline,
e.g., $\bf{i}^{\,\underline 3}=\bf{i}*\bf{i}*\bf{i}$}, and
$\bf{i}^{-\,\underline 1}$ indicates the inverse.  \end{con}

\begin{pro}\label{hesam}{In operator algebra $\mIm$ we have the following equalities: \bea &&(g(x,y)\theta(x-y))\ast
(f(x,y)\theta(x-y))=\big({\int^{x}_{y}}\, d z\,
g(x,z)f(z,y)\big)\theta(x-y),\label{p1}\nn\\
  \label{rr4}\nn
  && P(x) * f(x,y)\,\theta(x-y)=(P(x)f(x,y))\,\theta(x-y)
  ,\\
  \label{rr5}\nn
&& f(x,y)\,\theta(x-y) * P(x) = (P(y)f(x,y))\,\theta(x-y),\\
 \label{rr2}\nn
&&  \rd_{x} * f(x,y)\,\theta(x-y)= (\rd_{x}f(x,y))\,\theta(x-y)+f(x,x),\\
   \label{rr3}\nn
&& f(x,y)\,\theta(x-y) * \rd_{x} =
-(\rd_{y}f(x,y))\,\theta(x-y)+f(x,x),\\\label{tt1}\nn &&
P(x)*Q(x)=P(x)Q(x),
  \hspace{3cm} (\rd_{x})^{\,\underline n}*(\rd_{x})^{\,\underline m}=(\rd_{x})^{\,\underline {\,n+m}}=\rd^{\,n+m}_{x}\,,
   \\
  \label{tt3}\nn
&&  \rd_{x} * P(x) = (\rd_{x} P(x))+  P(x)*\rd_{x},  \hspace{1.65cm}
P(x)*\rd_{x}  = P(x)\rd_{x}.
  \eea}
\end{pro}

{\bf Proof.} The first equality is clear by Lemma \ref{lemf4}. Other
equalities can be verified easily, e.g., for the fifth one we have,
$\forall \phi(x)\in M[a,b]$, \bea&&(f(x,y)\,{\theta(x-y)}
*\rd_{x})\cdot\phi(x)=\int^{b}_{a} dz f(x,z)\theta(x-z)
\rd_{z}\phi(z) =\int^{x}_{a} dz f(x,z)
\rd_{z}\phi(z)\nn\\&&\hspace{1cm}=f(x,z)\phi(z)|^{x}_{a}-\int^{x}_{a}
dz (\rd_{z}f(x,z)) \phi(z)= f(x,x)\phi(x)-\int^{b}_{a} dz
(\rd_{z}f(x,z))\theta(x-z)\phi(z)\nn\\&&\hspace{1cm}=
f(x,x)\cdot\phi(x)-((\rd_{y}f(x,y))\,\theta(x-y))\cdot\phi(x).\nn
 \hspace{7cm}\text{\qed}\eea

\begin{cor}\label{iraj55} In operator algebra $\mIm$ we have, for
$n\ge1$, {\bea (\rd_{x}\,)^{\underline n}=\rd_{x}^{\,n},
\hspace{2cm} (\,\theta(x-y))^{\underline n}={(x-y)^{n-1}
\over(n-1)!}\,\,\theta(x-y),\nn\\(\rd_{x}\,)^{\underline
n}\,\,*\,\,(\,\theta(x-y))^{\underline n}=1,
\hspace{2cm}(\,\theta(x-y))^{\underline
n}\,\,*\,\,(\rd_{x}\,)^{\underline n}=1,\nn\eea \be
(\rd_{x}\,)^{-\underline n}\,\,=\,\,(\,\theta(x-y))^{\underline
n}\,\,=\,\,{(x-y)^{n-1} \over (n-1)!}\,\,\theta(x-y).\ee}
\end{cor}
Using multiplication rules in Proposition (\ref{hesam}), any ${\bf
i}\in \mIm$ can be written as: \be {\bf i}=\big(\sum^{n}_{i=1}
P_{i}(x) \,\rd_{x}^{\,i}+P_{0}(x)\big)+
 f(x,y)\,{\theta}(x-y),\label{point}\ee acting on $\phi(x)$ in
$M[a,b]$ by, \be {\bf i}\cdot \phi(x)=\sum^{n}_{i=1} P_{i}(x)
\,\rd_{x}^{\,i}\phi(x)+P_{0}(x)\phi(x)+\int^{\,b}_{a} dz\,
f(x,z)\theta(x-z)\,\phi(z). \label{acta}\nn\ee Now in the following
we prove several results by which we ultimately prove Theorem
\ref{emam} below.

\begin{pro}\label{sim1}{Let \,\,$h(x,y)\,{\theta}(x-y)\,\,\in \mIm$.
\\
1-There exits a unique element
 $R(x,y)\,{\theta}(x-y)\,\,\in \mIm$
such that \bea h(x,y)\,{\theta}(x-y)\,*
R(x,y)\,{\theta}(x-y)\,&=&R(x,y)\,{\theta}(x-y)\,*
h(x,y)\,{\theta}(x-y)\,\nn\\&=&
R(x,y)\,{\theta}(x-y)\,-h(x,y)\,{\theta}(x-y)\,.\label{lales11}\eea
2-The operator $(1-h(x,y)\theta(x-y))$ is invertible and
\be\big(1-h(x,y)\,{\theta}(x-y)\big)^{-\underline 1} =
1+R(x,y)\,{\theta}(x-y),\label{ali}\nn\ee where
$R(x,y)\,{\theta}(x-y)$ is given in {\em (\ref{lales11})}}.\end{pro}

{\bf Proof.} {Equation (\ref{lales11}) by using Proposition
\ref{hesam} can be rewritten as \be \Big(\int^{x}_{y} dz \,h(x,z)
R(z,y)\Big)\,{\theta}(x-y)=\Big(\int^{x}_{y} dz R(x,z) h(z,y)
\Big)\,{\theta}(x-y)\,=R(x,y)
\,{\theta}(x-y)\,-h(x,y)\,{\theta}(x-y)\,.\nn\ee Therefore the first
part of the Proposition can be proved if we show for any $h(x,y)\in
\boldsymbol{B_2}$, there exists a unique function
$R(x,y)\in\boldsymbol{B_2}$ for which, \be \int^{x}_{y} dz \,h(x,z)
R(z,y)=\int^{x}_{y} dz R(x,z) h(z,y)=R(x,y)-h(x,y).\label{ferri}\ee
The equation (\ref{ferri}) is well known in the context of Volterra
integral equation of second kind (e.g. see \cite{it}, section 1.2).
Specifically for every $h(x,y)\in C[\Delta]$ there exists $R(x,y)\in
C[\Delta]$, called {\em resolvent kernel}, satisfying (\ref{ferri})
and it is given by a uniformly convergent series, \be
R(x,y)=\lim_{n\to\infty}\sum^{n}_{i=1}h_{i}(x,y).\label{beygi1}\ee
where $h_{\,1}(x,y)=h(x,y)$ and \be h_{\,i}(x,y)=\int^{x}_{y}dz\,
h(x,z) \,h_{\,i-1}(z,y), \hspace{0.75cm}i\ge2\nn\ee Moreover if
$h(x,y)\in C^{k}[\Delta]$ then it follows that $R(x,y)\in
C^{k}[\Delta]$. Therefore we can conclude for every $h(x,y)\in
\boldsymbol{B_2}$ there exists $R(x,y)\in \boldsymbol{B_2}$
satisfying (\ref{ferri}). Thus the first part
is proved.\\
Part 2 follows from part 1. \qed\\

\begin{Rema}\label{zari}{\em {Recall that a sequence $\{T_{n}\}$ in
space of continuous linear maps on $M[a,b]$ converges to a
continuous operator $T$ if for every $k\in{\mathbb Z}^{+}$,
$\|T_{n}(\phi)-T(\phi)\|_{k}\overset{\mathrm{uniformly}}{\xrightarrow{\hspace*{0.75cm}}}0$}
for every $\phi\in M[a,b]$. Furthermore for any other continuous
operator $T{\,'}$ the sequences $\{T{\,'}*T_{n}\}$ and
$\{T_{n}*T{\,'}\}$ converge to $T{\,'}*T$ and $T*T{\,'}$
respectively. Now assume $\{R_{n}(x,y)\}$ is a sequence in
$\boldsymbol{B_2}$ and $R_n\to R$ in $C^{\infty}[\Delta]$ (all
partial derivatives converge uniformly on $\Delta$). Then the
sequence of operators $\{R_{n}(x,y)\theta(x-y)\}$ converges to the
operator $R(x,y)\theta(x-y)$.}\\
\end{Rema}

\begin{lem}\label{iraj45}{Suppose $P(x)$ and $P_{i}(x)\in \,\boldsymbol{B_{\,1}}$ for $i=0,1,\cdots,n$ and $f(x,y)\in
\boldsymbol{B_2}$. Then the following three types of operators
$\bf{i}$, $\bf{j}$, $\bf{k}$ $\in \mIm$ are invertible. \bea
{\bf{i}}&=& P(x),\hspace{0.2cm} where \hspace{0.2cm}  \forall
x\in[a,b],\,\, P(x)\ne 0\nn\,;\\
{\bf{j}}&=&P(x)-f(x,y)\theta(x-y),\hspace{0.2cm} where
\hspace{0.2cm} \forall x\in[a,b],\,\, P(x)\ne 0 \nn\,;
\\{\bf{k}}&=&\sum^{n}_{i=1} P_{i}(x) \,\rd_{x}^{\,i}+P_{0}(x)+
f(x,y)\theta(x-y),\hspace{0.2cm} where \hspace{0.2cm} \forall
x\in[a,b],\,\, P_{n}(x)\ne 0\,. \nn\eea}
\end{lem}

{\bf Proof}. {Clearly  ${(\,\bf{i}\,)}^{-\underline
1}=(P(x))^{-\underline 1}={1\over P(x)}$.\\
Now $ {\bf{j}}=P(x)-f(x,y)\theta(x-y)=P(x)*\Big(1-
(P(x))^{-\underline 1}*f(x,y)\theta(x-y)\Big)=P(x)*\Big(1-
{f(x,y)\over P(x)}\theta(x-y)\Big).$ Both $P(x)$ and $\Big(1-
{f(x,y)\over P(x)}\theta(x-y)\Big)$ are invertible, therefore
${\bf{j}}$ which is the product of two invertible operators is
itself invertible and we have, \be {(\bf{j})}^{-\underline 1}=
\Big(1- {f(x,y)\over P(x)}\theta(x-y)\Big)^{-\underline
1}*P(x)^{-\underline 1}.\nn\ee For operator ${\bf{k}}$, using
Proposition (\ref{iraj55}) we can write,\bea{\bf{k}}&=&\Big(
P_{n}(x)+\sum^{n-1}_{i=0} P_{i}(x)* \,(\rd_{x})^{-\underline
{(n-i)}}\,\,+ f(x,y)\theta(x-y)*(\rd_{x})^{-\underline
n}\Big)*(\rd_{x})^{\underline n}\nn \\ &=& P_{n}(x)*\Big(1
+\sum^{n-1}_{i=0} {P_{i}(x)\over P_{n}(x)} \,{(x-y)^{n-i-1}\over
{(n-i-1)!}}\theta(x-y)\,\,\nn\\&&+ {f(x,y)\over
P_{n}(x)}\theta(x-y)*{(x-y)^{n-1}\over
{(n-1)!}}\theta(x-y)\Big)*(\rd_{x})^{\underline n}\nn \,\,\,. \eea
Therefore ${\bf{k}}=
P_{n}(x)*\big(1-g(x,y)\theta(x-y)\big)*(\rd_{x})^{\underline n}$
where, \be g(x,y)=-\Big(\sum^{n-1}_{i=0} {P_{i}(x)\over P_{n}(x)}
\,{(x-y)^{n-i-1}\over {(n-i-1)!}}+ \int^{x}_{y} dz\,{f(x,z)\over
P_{n}(x)}{(z-y)^{n-1}\over {(n-1)!}}\Big).\nn\ee Since
$\bf{k}$ is the product of three invertible operators it is itself
invertible and we have \be \hspace{3cm}{(\bf{k})}^{-\underline
1}=(\rd_{x})^{-\underline n}*
\big(1-g(x,y)\theta(x-y)\big)^{-\underline 1}*P_{n}(x)^{-\underline
1}\,\,\,.\hspace{4cm \qed} \nn\ee The following comes as a
consequence.
\begin{cor}\label{iraj46}{\,If $P_{i}(x)\in \boldsymbol{B_1}$\, ($i=0,1,\cdots,n$), $f(x,y)\in
\boldsymbol{B_{2}}$ and $P_{n}(x)\neq 0$ for $x\in[a,b]$, then there
does not exist $0 \not\equiv  \phi(x)\in M[a,b]$ with, \be
\sum^{n}_{i=1} P_{i}(x)
\,\rd_{x}^{\,i}\phi(x)+P_{0}(x)\phi(x)+\int^{\,x}_{a} dz
f(x,z)\,\phi(z)=0.\label{chera}\nn\ee}
\end{cor}
{\bf Proof.} Since by Lemma \ref{iraj45} the operator $\bf{k}$ is
invertible, it is an automorphism on vector space $M[a,b]$.
Therefore $\ker{\bf{k}}=\{0\}$ which implies the result. \qed
\begin{lem}\label{jafar}{If $P_{0}(x)\in
\boldsymbol{B_{\,1}}$, $f(x,y)\in \boldsymbol{B_{\,2}}$ and \,\be
P_{0}(x) \phi(x)-\int^{\,x}_{a} dz f(x,z)\,\phi(z)=0, \hspace{0.5cm}
\forall\,\phi(x)\in M[a,b],\hspace{1.5cm}
(\,\star\,)\label{faramarz}\nn\ee then, $P_{0}(x)=0$ for $x\in[a,b]$
and $f(x,y)=0$ for $(x,y)\in \triangle$.}
\end{lem}

{\bf Proof.} Let us first show the following function is in function
space $M[a,b]$. Assuming $a\le c< d \le b$,
\be\psi_{\,c,\,d}(x)=\begin{cases}\exp({{-1\over{1-({c+d-2x\over{c-d}}
)^2}}})&\text{if}\,\, c< x <d,\\0 &\text{if} \,\,\,a\leq x \leq c,
\,\,\,\,\text{or}\,\,\,\, d\leq x \leq b.
\end{cases}\label{hossein}\ee One finds the $n^{th}$ derivative has the form
$\psi^{(n)}_{\,c,\,d}(x)=p_{\,n}(\,x,\,{1\over{1-({c+d-2x\over{c-d}}
)^2}})\psi_{\,c,\,d}(x)$ for $x\in(c,d)$ (where $p_{n}(s,t)$ is some
polynomial of $s$ and $t$). We also have $\psi^{(n)}_{\,c,\,d}(x)=0$
for $x\in[a,c]$ (if $a<c$) and $\psi^{(n)}_{\,c,\,d}(x)=0$ for
$x\in[d,b]$ (if $d<b$), thus one concludes $\psi^{(n)}_{\,c,\,d}(x)$
is continuous for all points in $[a,b]$\textbackslash$\{c,d\}$. But
also the function $\psi^{(n)}_{\,c,\,d}(x)$ is continuous at $x=c$
and $x=d$ as one finds, \bea \lim_{x\to
c^+}\psi^{(n)}_{\,c,\,d}(x)=\lim_{x\to c^+}
p_{n}(\,x,\,{1\over{1-({c+d-2x\over{c-d}}
)^2}})\psi_{\,c,\,d}(x)=0,\nn\\\lim_{x\to
d^-}\psi^{(n)}_{\,c,\,d}(x)=\lim_{x\to d^-}
p_{n}(\,x,\,{1\over{1-({c+d-2x\over{c-d}} )^2}})\psi_{\,c,\,d}(x)=0.
\nn\eea So clearly the function $\psi_{\,c,\,d}(x)$ is infinitely
differentiable for $x\in[a,b]$ and also
$\psi_{\,c,\,d}(a)=\psi_{\,c,\,d}(b)=0$, ${d^{n}\over
dx^{n}}\psi_{\,c,\,d}(a)={d^{n}\over dx^{n}}\psi_{\,c,\,d}(b)=0$ for
all $n\ge1$, thus $\psi_{\,c,\,d}(x)\in M[a,b]$. We also note
$\psi_{\,c,\,d}(x)>0$ for $x\in(c,d)$. \\ Now we set, \be
\phi(x)=\overline{f(d,x)}\,\,\psi_{c,d}(x),\nn\ee where
$\overline{f(d,x)}$ is complex conjugate of $f(d,x)$. By putting
$\phi(x)$ into $(\star)$ we have; \be \int^{\,x}_{c} dz
f(x,z)\,\overline{f(d,z)}\,\,\,\psi_{\,c,\,d}(z) =P_{0}(x)
\,\overline{f(d,x)}\,\,\,\psi_{\,c,\,d}(x),\nn\ee where in the left
side $a$ is replaced by $c$ because $\psi_{\,c,\,d}(x)=0$ for $x\le
c$. In above relation, for $x=d$ we have,

\be \int^{\,d}_{c} dz |f(d,z)|^{\,2}\,\psi_{\,c,\,d}(z) =P_{0}(d)
\,\,\,\overline{f(d,d)}\,\,\,\psi_{\,c,\,d}(d)=0,\label{ferferi}\nn\ee
because $\psi_{\,c,\,d}(d)=0$. The left hand side of above relation
is a constant real positive number which is a contradiction unless
$f(d,z)=0$ for $c\le z \le d$. We can repeat the procedure for any
interval $[c,d]\subseteq[a,b]$ and conclude $f(x,y)=0$ for $b \geq
x\geq y\geq a$. Now in $(\star)$ we are left with
$P_{0}(x)\phi(x)=0$, $\forall\, \phi(x)\in M[a,b]$.
Then trivially $P_{0}(x)\equiv 0$.\qed \\

\begin{thm}\label{emam}{If $P_{i}(x)\in \,\boldsymbol{B_{\,1}}$ for $i=0,1,\cdots,n$\, and $f(x,y)\in
\boldsymbol{B_2}$ with\, \be\sum^{n}_{i=1} P_{i}(x)
\,\rd_{x}^{\,i}\phi(x)+P_{0}(x)\phi(x)+\int^{\,x}_{a} dz
f(x,z)\,\phi(z)=0,\hspace{0.5cm}\forall\, \phi(x)\in
M[a,b],\label{baghi}\nn\hspace{1.5cm} (\,\dagger\,)\ee
then\,$P_{i}(x)\equiv 0$ for $i=0,1,\cdots,n$ and $f(x,y)=0$ for
$(x,y)\in \triangle$.}
\end{thm}

{\bf Proof.} Assume $P_{n}(x)\not\equiv 0$. By continuity there exists
$[r_1,r_2]\subseteq[a,b]$ such that \be
P_{n}(x)\ne0\hspace{0.5cm} \text{for} \hspace{0.5cm}x\in
[r_1,r_2].\nn\ee
For interval $[r_1,r_2]$, consider the function $\psi_{r_1,r_2}(x)$
defined as in (\ref{hossein}), where $c$ and $d$ are replaced with
$r_1$ and $r_2$. By assumption from $(\dagger)$ we have for
$x\in[r_1,r_2]$: \be\sum^{n}_{i=1} P_{i}(x)
\,\rd_{x}^{\,i}\psi_{r_1,r_2}(x)+P_{0}(x)\psi_{r_1,r_2}(x)+\int^{\,x}_{r_{1}}
dz, f(x,z)\,\psi_{r_1,r_2}(z)=0,\label{daj}\ee where in the integral
above we have replaced the real number\, $a$\, with real number
\,$r_{1}$ \,since $\psi_{r_1,r_2}(x)=0$ for $a\leq x\leq r_{1}$. But
(\ref{daj}) is in contradiction with Corollary \ref{iraj46} applied
to algebra $\widehat{I}_{M{\,'}}$ (defined in complete analogy with
$\mIm$ but as an algebra of operators acting on function space
$M[r_{1}, r_{2}]$). Therefore $P_{n}(x)\equiv0$. Repeating the same argument for the
highest index $i$ with $P_i\not\equiv0$ gives $P_{i}(x)\equiv0$ for
$i=1,2,\cdots,n$. The problem is now reduced to the case of Lemma
\ref{jafar} by which we conclude $f(x,y)= 0$ for $(x,y)\in\triangle$
and $P_{0}(x)\equiv 0$ .\qed\\

\begin{cor}\label{kor}{The expression (\ref{point}) is unique in the sense that all $P_{i}(x)$ and $f(x,y)$
are uniquely determined by ${\bf i}\in \mIm$}
\end{cor}
 \section {Algebra $\mIs$}
 \label{mary}
In order to distinguish the different parts of an operator in
$\mIm$, we will find a particular representation $\mIs$ as follows.
It will then be more convenient to work in $\mIs$.

Let $\mathcal{D}$ be the free left-$\boldsymbol{B_1}$ module on the
set $\{1,\rd_{x},\rd^{2}_{x}\cdots\}$ inside $\mIm$. By Lemma
\ref{jafar} clearly  $\mathcal{G}$ is free as a
left-$\boldsymbol{B_2}$ module. Set
$\mIs:=\mathcal{D}\oplus\mathcal{G}$ (and note that according to
Lemma \ref{lemf4} and Proposition \ref{hesam}, $\mathcal{D}$ and
$\mathcal{G}$ are closed under composition). Clearly there is a
$\mc$-linear
isomorphism $\xi:\mIm\to\mathcal{D}\oplus\mathcal{G}$ so that,\\
$ \displaystyle \text{if}\hspace{0.3cm} {\bf i}=\big(\sum^{n}_{i=1}
P_{i}(x) \,\rd_{x}^{\,i}+P_{0}(x)\big)+
 f(x,y)\,{\theta}(x-y)\in\mIm, \hspace{0.3cm} \text{then} \hspace{0.3cm}  \xi({\bf i})=\big(\sum^{n}_{i=1} P_{i}(x)
\,\rd_{x}^{\,i}+P_{0}(x)\big)\,\widehat{\beta}+
f(x,y)\,\widehat{\theta},\nn$\\
 where \,$\widehat{\beta}=(1,0)$,
$\,\widehat{\theta}=(0,1)$ are unit vectors. Using $\xi$, the space
$\mIs$ enjoys a $\mc$-algebra structure with identity
$\widehat{\beta}$ in which the multiplication denoted by $\bullet$
satisfies the following rules:
 \bea  f(x,y)
\,\widehat{\theta}\bullet g(x,y)\,\widehat{\theta}\, &=& (
\int^{x}_{y} f(x,z) \,g(z,y)\, dz\,)\,\widehat{\theta}, \nn\label{ss4}\\
  P(x)\,\widehat{\beta} \bullet f(x,y)\,\widehat{\theta}\,&=& (P(x)f(x,y))\,\widehat{\theta}\,,
  \label{ss5}\nn\\
 f(x,y)\,\widehat{\theta} \bullet P(x)\,\widehat{\beta} &=& (P(y)f(x,y))\,\widehat{\theta},\label{ss2}\nn\\
  \rd_{x}\,\widehat{\beta} \bullet f(x,y)\,\,\widehat{\theta}\,&=& (\rd_{x}f(x,y))\,\widehat{\theta}\,+f(x,x)\widehat{\beta},
    \label{ss3}\nn\\
 f(x,y)\,\,\widehat{\theta}\, \bullet \rd_{x}\,\widehat{\beta} &=&
-(\rd_{y}f(x,y))\,\widehat{\theta}+f(x,x)\,\widehat{\beta},\nn\label{ttss1}\\
(\rd_{x}\,\widehat{\beta})^{\, n}\bullet
 (\rd_{x}\,\widehat{\beta})^{\,m}&=&(\rd_{x}\,\widehat{\beta})^{\,{n+m}}
 =\rd^{\,n+m}_{x}\,\widehat{\beta}, \hspace{0.75cm}
 P(x)\,\widehat{\beta}\bullet Q(x)\,\widehat{\beta} =
 P(x)Q(x)\,\widehat{\beta},
 \nn\\
  \label{ttss3}
  \rd_{x} \,\widehat{\beta}\bullet P(x)\,\widehat{\beta} &=& (\rd_{x}
P(x))\,\widehat{\beta}
 + P(x)\,\widehat{\beta}\bullet\rd_{x}\,\widehat{\beta}, \hspace{0.75cm}
P(x)\,\widehat{\beta}\bullet\rd_{x}\,\widehat{\beta}  =
P(x)\rd_{x}\,\widehat{\beta}.\label{ttss4}\nn\eea The above
consideration yields:
\begin{thm}\label{ershi}{The algebras $\mIs$ and $\mIm$ are isomorphic.}
\end{thm}
\begin{Rema}\label{sublm}{\em{The sets $S_1=\{P(x)\,\widehat{\beta}\,| \,P(x)\in
\boldsymbol {B_1}\}$, $S_2=\{\big(\sum^{n}_{i=1} P_{i}(x)
\,\rd_{x}^{\,i}+P_{0}(x)\big)\,\widehat{\beta}\,| \,P_i(x)\in
\boldsymbol {B_1}, \,i=0,1,...,n;\, n\in \Na\}$,
$S_3=\{P(x)\,\widehat{\beta}+f(x,y)\,\widehat{\theta}\,| \,P(x)\in
\boldsymbol {B_1},\,f(x,y)\in\boldsymbol {B_2}\}$ and
$S_4=\{f(x,y)\,\widehat{\theta}\,| \,f(x,y)\in \boldsymbol {B_2}\}$
are four different subalgebras of $\mIs$. Furthermore the following
relations hold
 \be (\rd_{x}\,\widehat{\beta})^{n}=\rd_{x}^{\,n}\,\widehat{\beta},
\hspace{0.6cm} (\,\widehat{\theta})^{\,n}={(x-y)^{n-1} \over
(n-1)!}\,\,\widehat{\theta},\hspace{0.6cm}
(\rd_{x}\,\widehat{\beta})^{n}\bullet
 (\,\widehat{\theta})^{\,n}= (\,\widehat{\theta})^{\,n}\bullet
(\rd_{x}\,\widehat{\beta})^{n}=\,\widehat{\beta}.\nn\label{barezi}\ee}}
\end{Rema}
\begin{definition}\label{khalil}{ Let $\Psi=\big(\sum^{n}_{i=1} P_{i}(x)
\,\rd_{x}^{\,i}+P_{0}(x)\big)\,\widehat\beta\,+g(x,y)\widehat{\theta}$.
We define projection operators \be
\mathcal{\pi}_{\widehat\beta}[\Psi]=\big(\sum^{n}_{i=1} P_{i}(x)
\,\rd_{x}^{\,i}+P_{0}(x)\big)\,\widehat\beta, \hspace{1.5cm}
  {\pi}_{\widehat{\theta}}[\Psi]=g(x,y)\widehat{\theta}.\nn\ee The degree function and coefficient operator are defined
  by, \be
{\Deg}[\Psi]=\begin{cases} n & \mbox { if } P_{n}(x)\not\equiv 0 \displaystyle \\
0  & \mbox { if } P_{i}(x)\equiv 0\, for\, $i=1,...,n$.
\end{cases} \hspace{0.5cm}\Cof[\Psi]=\begin{cases}\, P_{n}(x) & \mbox { if } P_{n}(x)\not\equiv 0 \\
P_{0}(x)  & \mbox { if } P_{i}(x)\equiv 0\, for\, $i=1,...,n$.
\end{cases} \nn\ee\\}
\end{definition}
The following lemma is immediate.\\
\begin{lem}\label{salar2}{1-Let $\Psi$, $\Phi\in\mIs$ with
\, ${\rm Coef[\Phi]}{\rm Coef[\Psi]}\not\equiv 0$. Then \be
{\Deg}[\Psi\bullet\Phi]= {\rm
deg_{\widehat{\beta}}}[\Phi\bullet\Psi]= {\rm
deg_{\widehat{\beta}}}[\Psi]+{\Deg}[\Phi], \hspace{1cm} {\rm
Coef}[\Psi\bullet\Phi]=
 {\rm Coef}[\Phi\bullet\Psi]=  {\rm Coef}[\Phi] {\rm Coef}[\Psi].\nn\ee 2-Let
 $\Psi\in\mIs$ with
\,${\Deg}[\Psi]\geq n\geq 1$ and $\Phi=h(x){\widehat{\beta}}\bullet
(\widehat{\theta})^n\bullet q(x){\widehat{\beta}}$ with ${\rm
Coef[\Psi]}\, h(x)\, q(x) \not\equiv 0$.\, Then \be {\rm
deg_{\widehat{\beta}}}[\Psi\bullet\Phi]= {\rm
deg_{\widehat{\beta}}}[\Phi\bullet\Psi]= {\rm
deg_{\widehat{\beta}}}[\Psi]-n, \hspace{1cm} {\rm
Coef}[\Psi\bullet\Phi]=
 {\rm Coef}[\Phi\bullet\Psi]=  h(x) q(x) {\rm Coef}[\Psi].\nn\ee}
\end{lem}

\begin{lem}\label{mahmood}{For every $K(x,y)\widehat{\theta}$ and
\,${\mathcal{Q}_{\,n}}(x,\rd_{x})\,\widehat{\beta}$ in $\mIs$ where
${\rm
deg_{\widehat{\beta}}}[{\mathcal{Q}_{\,n}}(x,\rd_{x})\,\widehat{\beta}]=n$,
we have: \be {\rm
deg_{\widehat{\beta}}}\,[{\mathcal{Q}_{\,n}}(x,\rd_{x})\,\widehat{\beta}\bullet
K(x,y)\widehat{\theta}\,\,]< n, \hspace{0.5cm} {\rm
deg_{\widehat{\beta}}}\,[
K(x,y)\widehat{\theta}\bullet{\mathcal{Q}_{\,n}}(x,\rd_{x})\,\widehat{\beta}\,\,]<
n. \label{rahamin1}\nn \ee}
\end{lem}
{\bf Proof.}  We prove this by showing $\forall P(x)\in \boldsymbol
{B_{1}}$ we have the following relations, \be (a)\hspace{0.5cm} {\rm
deg_{\widehat{\beta}}}\,[(P(x)\rd^{\,n}_{x})\,\widehat{\beta}\bullet
K(x,y)\widehat{\theta}]< n\,, \hspace{0.5cm}(b)\hspace{0.5cm} {\rm
deg_{\widehat{\beta}}}\,[ K(x,y)\widehat{\theta}\bullet
(P(x)\rd^{\,n}_{x})\,\,\widehat{\beta}]< n. \label{raha1}\nn\ee

These relations are proved by induction. For $n=1$ we have
\be(P(x)\rd_{x})\,\widehat{\beta}\bullet
K(x,y)\widehat{\theta}=P(x)\,\widehat{\beta}\bullet(K(x,x)\,\widehat{\beta}+\rd_{x}
K(x,y)\widehat{\theta})=((P(x)K(x,x))\,\widehat{\beta}+P(x)\rd_{x}
K(x,y)\widehat{\theta})\nn\ee Therefore $ {\rm
deg_{\widehat{\beta}}}\,[(P(x)\rd_{x})\,\widehat{\beta}\bullet
K(x,y)\widehat{\theta}]= 0. $\\
Now assume that (a) is correct for $n=k$. By taking $P(x)=1$ we
have, ${\rm
deg_{\widehat{\beta}}}\,[(\rd^{\,k}_{x})\,\widehat{\beta}\bullet
K(x,y)\widehat{\theta}]< k$, therefore,
$(\rd^{\,k}_{x})\,\widehat{\beta}\bullet
K(x,y)\widehat{\theta}={\mathcal{Q}_{\,l}}(x,\rd_{x})\widehat{\beta}+R(x,y)\widehat{\theta}$
where ${\rm
deg_{\widehat{\beta}}}[{\mathcal{Q}_{\,l}}(x,\rd_{x})\,\widehat{\beta}]=l<k$.
Thus,

\bea (P(x)\,\rd^{\,k+1}_{x})\,\widehat{\beta}\bullet
K(x,y)\widehat{\theta}&=&
(P(x)\,\rd_{x})\,\widehat{\beta}\bullet(\rd^{\,k}_{x})\,\widehat{\beta}\bullet
K(x,y)\widehat{\theta}\nn\\&=&
(P(x)\,\rd_{x})\,\widehat{\beta}\bullet
\big({\mathcal{Q}_{\,l}}(x,\rd_{x})\,\widehat{\beta}\,+\,R(x,y)\widehat{\theta}\,\,\big)
\nn\\&=&{\mathcal{Q}_{\,l+1}}(x,\rd_{x})\,\widehat{\beta}+(P(x)R(x,x))\,\widehat{\beta}+(P(x)\rd_{x}
R(x,y))\widehat{\theta},\nn\eea where
${\mathcal{Q}_{\,l+1}}(x,\rd_{x})\,\widehat{\beta}=(P(x)\,\rd_{x})\,\widehat{\beta}\bullet
{\mathcal{Q}_{\,l}}(x,\rd_{x})\,\widehat{\beta}$. From the last line
above we get, \be {\rm
deg_{\widehat{\beta}}}\,[(P(x)\rd^{\,k+1}_{x})\,\widehat{\beta}\bullet
K(x,y)\widehat{\theta}]\le l+1<k+1.\nn\ee The conclusion holds for
${\mathcal{Q}_{\,n}}(x,\rd_{x})\,\widehat{\beta}$. The relation (b)
can be proved similarly.{\qed} \\

By Lemma \ref{mahmood} it is obvious that if ${\rm
deg_{\widehat{\beta}}}[{\mathcal{Q}_{\,n}}(x,\rd_{x})\,\widehat{\beta}]=n$
then for any $K(x,y)\widehat{\theta}$ in $\mIs$ we have: \be {\rm
deg_{\widehat{\beta}}}\,[{\mathcal{Q}_{\,n}}(x,\rd_{x})\,\widehat{\beta}\bullet
(K(x,y)\widehat{\theta})^{n}]= 0, \hspace{0.5cm} {\rm
deg_{\widehat{\beta}}}\,[
(K(x,y)\widehat{\theta})^n\bullet{\mathcal{Q}_{\,n}}(x,\rd_{x})\,\widehat{\beta}]=
0.  \label{rahamin2}\nn\ee

Now we state the following result which is a consequence of
Proposition \ref{sim1}; the standard Volterra estimates give
$C^{\infty}[\Delta]$ convergence of the resolvent series, so Remark
\ref{zari} applies.

\begin{pro}\label{sim}{1-For any element $h(x,y)\,\widehat{\theta}\,\,\in\mIs$ there exits unique element
$R(x,y)\,\widehat{\theta}\in\mIs$ such that \be
h(x,y)\,\widehat{\theta}\bullet
R(x,y)\,\widehat{\theta}=R(x,y)\,\widehat{\theta}\bullet
h(x,y)\,\widehat{\theta}=R(x,y)\,\widehat{\theta}-h(x,y)\,\widehat{\theta},\label{lales}\ee
and $R(x,y)\,\widehat{\theta}$ is given by: \be
R(x,y)\,\widehat{\theta} =\sum^{\infty}_{n=1}
\,\,(h(x,y)\,\widehat{\theta}\,)^{\, n}.\label{iraj}\ee The element
$(\,\widehat{\beta}-h(x,y)\,\widehat{\theta})$ in algebra $\mIs$ is
invertible, \be (\widehat{\beta}-h(x,y)\,\widehat{\theta})^{-1} =
\widehat{\beta}+R(x,y)\,\widehat{\theta}=\widehat{\beta}+
\sum^{\infty}_{n=1} {\,\,(h(x,y)\,\widehat{\theta}\,)}^{\,
 n}.\label{taghdim}\ee}\end{pro}

Stated otherwise, equation (\ref{lales}) says $\sum_{n=1}^{\infty}
\Phi^{n}\in S_4$  for $\Phi\in S_{4}$ and \bea
\Phi\bullet(\sum_{n=1}^{\infty} \Phi^{n})=(\sum_{n=1}^{\infty}
\Phi^{n})\bullet
\Phi=\sum_{n=2}^{\infty}\Phi^{n}=(\sum_{n=1}^{\infty}\Phi^{n})-\Phi.\nn\\\nn\\\nn\eea

For convenience we set $ \PP=\{P(x)\in {\boldsymbol{B_1}}
\,\,|\,\,P(x)\ne 0\,\, for\,\, x\in [a,b]\,\}.$

\begin{lem}\label{hasan}{The following types of elements are invertible in
$\mIs$.\\
1.\,\, All elements $P_{0}(x)\,\widehat{\beta}$ where $P_{0}(x)\in\PP$,\\
2.\,\, All elements $P_{0}(x)\,\widehat{\beta}+f(x,y)\,\widehat{\theta}$ where $P_{0}(x)\in\PP$ and $f(x,y)\,\widehat{\theta}\ne 0$,\\
{3.\,\,All elements $\big(\sum^{n}_{i=1} P_{i}(x)
\,\rd_{x}^{\,i}+P_{0}(x)\big)\,\widehat{\beta}+
 f(x,y)\,\widehat{\theta}$, where  $P_{n}(x)\in\PP$  and $n\ge1$.}
}\end{lem}

{\bf Proof.} For a type 1 element we trivially have
$\displaystyle(P_{0}(x)\,\widehat{\beta})^{-1}=({1\over
P_{0}(x)})\,\widehat{\beta}$.\\
For a type 2 element, since \,$\displaystyle
\Psi=P_{0}(x)\,\widehat{\beta}+f(x,y)\,\widehat{\theta}=P_{0}(x)\,\widehat{\beta}\bullet(\,\widehat\beta-{-f(x,y)\over
P_{0}(x)}\,\widehat{\theta}\,)$, we have the inverse as
\be\Psi^{-1}= (\,\widehat\beta+\sum^{\infty}_{n=1} ({-f(x,y)\over
P_{0}(x)}\,\widehat{\theta}\,)^{n})\bullet({1\over
P_{0}(x)})\,\widehat{\beta}\,\,=\,\,({1\over
P_{0}(x)}\,\widehat{\beta}) +\big(\sum^{\infty}_{n=1} ({-f(x,y)\over
P_{0}(x)}\,\widehat{\theta}\,)^{\, n}\,\big)\bullet({1\over P_{0}(x)})\,\widehat{\beta}.\nn\ee\\
For a type 3 element $\Phi=\big(\sum^{n}_{k=1} P_{k}(x)
\rd_{x}^{\,k}+P_{0}(x)\big)\,\widehat{\beta}+
 f(x,y)\,\widehat{\theta}$  we write
\be \Phi=\, P_{n}(x)\,\widehat{\beta}\bullet\big(\,\widehat\beta
-\sum_{k=0}^{n-1} \,{-P_{k}(x)\over
P_{n}(x)}\widehat{\beta}\bullet(\,\widehat{\theta}\,)^{\,
{n-k}}\,-{-f(x,y)\over
P_{n}(x)}\,\widehat{\theta}\,\,\bullet\,(\,\widehat{\theta}\,)^{n}\big)
\,\bullet \,{\rd^{\,n}_{x}}\,\widehat{\beta},\nn\ee
\bea\Phi^{-1}&=&(\,\widehat{\theta}\,)^{\, n}
 \bullet\big(\,\widehat\beta+ \sum_{r=1}^{\infty}\big(\sum_{k=0}^{n-1} \,{-P_{k}(x)\over
P_{n}(x)}\widehat{\beta}\bullet\,(\,\widehat{\theta}\,)^{\,
{n-k}}\,+{-f(x,y)\over
P_{n}(x)}\,\widehat{\theta}\,\,\bullet\,(\,\widehat{\theta}\,)^{\,
n} \big)^{\,r}\big)\bullet
({P_{n}(x)\,\widehat{\beta}})^{-1}\nn\\&=&\widehat{\theta}^{\,
n}\bullet({1\over
P_{n}(x)})\,\widehat{\beta}+\,\widehat{\theta}\,^{\, n}
\bullet\Big(\sum_{r=1}^{\infty}\big(\sum_{k=0}^{n-1}
\,{-P_{k}(x)\over
P_{n}(x)}\widehat{\beta}\bullet\,(\,\widehat{\theta}\,)^{\,\,
{n-k}}\,+{-f(x,y)\over
P_{n}(x)}\,\widehat{\theta}\,\bullet(\,\widehat{\theta}\,)^{\,n}
\big)^{\,r}\Big)\bullet ({1\over
P_{n}(x)})\,\widehat{\beta}.\nn\qed\eea

\begin{definition}\label{albert}{Let  $\mathcal U$ denote the subgroup generated by elements of types 1,2 and 3 mentioned in Lemma \ref{hasan}.
Also define : \be W_1=\{\Psi\in \mIs\,|\,\,\, {\rm
deg_{\widehat{\beta}}}[\Psi]=0; {\rm Coef[\Psi]}\in \PP \}, \nn\ee
\be W_2=\{\Psi\in \mIs\,|\,\,\, {\rm
deg_{\widehat{\beta}}}[\Psi]\ge1; {\rm Coef[\Psi]}\in\PP \}, \nn\ee
\be W_3=\{\Psi\in \mIs\,|\,\,\, {\rm
deg_{\widehat{\beta}}}[\Psi^{-1}]\ge1; {\rm Coef[\Psi^{-1}]}\in\PP
\}. \nn\ee }
\end{definition}

Note that type 1 and 2 generators of ${\mathcal U}$ and their
inverses are elements of $W_1$. Generators of type 3 are elements of
$W_2$ while their inverses are elements of $W_3$.\\

\begin{pro}\label{salar}{ With the above notation ${\mathcal U}=W_1\cup W_2 \cup W_3 $.\\}
\end{pro}

{\bf Proof.} It suffices to prove that $W_1\cup W_2 \cup W_3 $ is
closed under multiplication. We proceed by considering different
cases and use Lemma (\ref{salar2}) frequently.\\

\begin{itemize}
{\item[]{1.\hspace{0.5cm} For $\Psi$, $\Phi\in W_1$ we have, \be
{\rm deg_{\widehat{\beta}}}[\Psi\bullet\Phi]= {\rm
deg_{\widehat{\beta}}}[\Phi\bullet\Psi]= 0, \hspace{0.6cm} {\rm
Coef}[\Psi\bullet\Phi]=
 {\rm Coef}[\Phi\bullet\Psi]=  {\rm Coef}[\Phi] {\rm Coef}[\Psi]\in\PP,\nn\ee
and therefore in this case, $ \Psi\bullet\Phi \in W_1$ and
$\Phi\bullet\Psi \in W_1$. This shows that $W_1$ is a subgroup of
$\mathcal U$.}

\item[]{2.\hspace{0.5cm} For $\Psi \in W_1$, $\Phi\in W_2$ we have, \be {\rm
deg_{\widehat{\beta}}}[\Psi\bullet\Phi]= {\rm
deg_{\widehat{\beta}}}[\Phi\bullet\Psi]= {\rm
deg_{\widehat{\beta}}}[\Phi]\geq  1,\hspace{0.6cm} {\rm
Coef}[\Psi\bullet\Phi]=
 {\rm Coef}[\Phi\bullet\Psi]=  {\rm Coef}[\Phi] {\rm Coef}[\Psi]\in\PP,\nn\ee
  and therefore in this case, $ \Psi\bullet\Phi \in W_2$ and $\Phi\bullet\Psi \in
 W_2$.}

\item[]{3.\hspace{0.5cm}  For $\Psi \in W_1$, $\Phi\in W_3$ from the
 fact that $W_1$ is a subgroup of $\mathcal U$ and by definition of $W_3$ we have $(\Psi)^{-1} \in W_1$, $(\Phi)^{-1}\in W_2$ and  \be {\rm
deg_{\widehat{\beta}}}[\Psi^{-1}\bullet\Phi^{-1}]= {\rm
deg_{\widehat{\beta}}}[\Phi^{-1}\bullet\Psi^{-1}]= {\rm
deg_{\widehat{\beta}}}[\Phi^{-1}]\geq  1,\nn\ee \be {\rm
Coef}[\Psi^{-1}\bullet\Phi^{-1}]=
 {\rm Coef}[\Phi^{-1}\bullet\Psi^{-1}]=  {\rm Coef}[\Phi^{-1}] {\rm Coef}[\Psi^{-1}]\in \PP.\nn\ee
Therefore, \be {\rm deg_{\widehat{\beta}}}[(\Psi\bullet\Phi)^{-1}]=
{\rm deg_{\widehat{\beta}}}[(\Phi\bullet\Psi)^{-1}]= {\rm
deg_{\widehat{\beta}}}[\Phi^{-1}]\geq  1,\nn\ee \be {\rm
Coef}[(\Psi\bullet\Phi)^{-1}]=
 {\rm Coef}[(\Phi\bullet\Psi)^{-1}]=  {\rm Coef}[\Phi^{-1}] {\rm Coef}[\Psi^{-1}]\in\PP.\nn\ee
So in this case, $ \Psi\bullet\Phi\in W_3$ and $\Phi\bullet\Psi \in
 W_3$.}\\

\item[]{4.\hspace{0.5cm}  For $\Psi$, $\Phi\in W_2$ we have,  \be {\Deg} [\Psi\bullet\Phi]=
{\rm deg_{\widehat{\beta}}}[\Phi\bullet\Psi]= {\rm
deg_{\widehat{\beta}}}[\Psi]+ {\Deg}[\Phi]\geq 2, \hspace{0.5cm}
{\rm Coef}[\Psi\bullet\Phi]=
 {\rm Coef}[\Phi\bullet\Psi]=  {\rm Coef}[\Phi] {\rm Coef}[\Psi]\in\PP.\nn\ee Thus, $ \Psi\bullet\Phi \in W_2$ and
$\Phi\bullet\Psi \in
 W_2$.\\

\item[]5.\hspace{0.5cm}  Let $\Psi \in W_2$ and $\Phi\in W_3$ with ${\Deg}[\Psi]=m$ and
${\Deg}[\Phi^{-1}]= n$  (where $m,n\geq1$).

(a) Assume $m>n$.

By assumption, $\Phi^{-1}=\Big(\big(\sum^{n}_{k=1} P_{k}(x)
\rd_{x}^{\,k}+P_{0}(x)\big)\,\widehat{\beta}+
 f(x,y)\,\widehat{\theta}\Big)$ where $P_{n}(x)\in\PP$. By Lemma \ref{hasan},
\be\Phi=\Big(\big(\sum^{n}_{k=1} P_{k}(x)
\rd_{x}^{\,k}+P_{0}(x)\big)\,\widehat{\beta}+
 f(x,y)\,\widehat{\theta}\Big)^{-1}=\widehat{\theta}^{\,
n}\bullet({1\over
P_{n}(x)})\,\widehat{\beta}+\,\widehat{\theta}\,^{\,
n}\bullet\,R(x,y)\,\widehat{\theta},\nn\ee where $
R(x,y)\,\widehat{\theta}=\Big(\sum_{r=1}^{\infty}\big(\sum_{k=0}^{n-1}
\,{-P_{k}(x)\over
P_{n}(x)}\widehat{\beta}\bullet(\,\widehat{\theta}\,)^{\,\,
{n-k}}\,+{-f(x,y)\over
P_{n}(x)}\,\widehat{\theta}\,\bullet(\,\widehat{\theta}\,)^{\,n}
\big)^{\,r}\Big)\bullet ({1\over P_{n}(x)})\,\widehat{\beta}$. Now
by Lemma \ref{salar2} and Lemma \ref{mahmood},

\be {\rm deg_{\widehat{\beta}}}[\Psi\bullet\big(\widehat{\theta}^{\,
n}\bullet({1\over P_{n}(x)}){\widehat{\beta}}\,\big)]= {\rm
deg_{\widehat{\beta}}}[\big(\widehat{\theta}^{\, n}\bullet({1\over
P_{n}(x)}){\widehat{\beta}}\big)\bullet\Psi]= m-n,
\label{karbas1}\nn\ee

\be {\rm
deg_{\widehat{\beta}}}[\Psi\bullet\big(\,\widehat{\theta}\,^{\,
n}\bullet\,R(x,y)\,\widehat{\theta}\big)]< m-n,\hspace{1cm} {\rm
deg_{\widehat{\beta}}}[\big(\,\widehat{\theta}\,^{\,
n}\bullet\,R(x,y)\,\widehat{\theta}\big)\bullet \Psi]<
m-n.\label{karbas2}\nn\ee Therefore from last two lines we deduce,
\be {\Deg}[\Psi\bullet\Phi]={\Deg}[\Phi\bullet\Psi]=m-n>0,\nn\ee \be
{\rm Coef}[\Psi\bullet\Phi]=
 {\rm Coef}[\Phi\bullet\Psi]={\Cof}[\Psi\bullet\big(\widehat{\theta}^{\,
n}\bullet({1\over P_{n}(x)}){\widehat{\beta}}\,\big)]=
{\Cof}[\big(\widehat{\theta}^{\, n}\bullet({1\over
P_{n}(x)}){\widehat{\beta}}\,\big)\bullet\Psi],\nn\ee  \be {\rm
Coef}[\Psi\bullet\Phi]=
 {\rm Coef}[\Phi\bullet\Psi]={1\over P_{n}(x)}{\rm Coef}[\Psi]\in \PP.\nn \ee
Thus $\Psi\bullet\Phi \in W_2$ and $\Phi\bullet\Psi \in \nn W_2$.\\

(b) Assume $m=n$ and use a similar procedure as before to get \be
{\Deg}[\Psi\bullet\Phi]={\Deg}[\Phi\bullet\Psi]=m-n=0,\hspace{0.75cm}
{\rm Coef}[\Psi\bullet\Phi]=
 {\rm Coef}[\Phi\bullet\Psi]={1\over P_{n}(x)}{\rm Coef}[\Psi]\in \PP.\nn \ee Hence $ \Psi\bullet\Phi \in W_1$ and $\Phi\bullet\Psi \in
 W_1$.\\

(c) Let $m<n$ and write

$\Phi^{-1}=\Big(\big(\sum^{n}_{k=1} P_{k}(x)
\rd_{x}^{\,k}+P_{0}(x)\big)\,\widehat{\beta}+
 g(x,y)\,\widehat{\theta}\Big)$ where $P_{n}(x)\in \PP$
and \\ $\Psi=\Big(\big(\sum^{m}_{k=1} Q_{k}(x)
\rd_{x}^{\,k}+Q_{0}(x)\big)\,\widehat{\beta}+
 f(x,y)\,\widehat{\theta}\Big)$ where $Q_{m}(x)\in \PP $. From
Lemma \ref{hasan}, \be \Psi^{-1}=\Big(\big(\sum^{m}_{k=1} Q_{k}(x)
\rd_{x}^{\,k}+Q_{0}(x)\big)\,\widehat{\beta}+
 f(x,y)\,\widehat{\theta}\Big)^{-1}=\widehat{\theta}^{\,
m}\bullet({1\over
Q_{m}(x)})\,\widehat{\beta}+\,\widehat{\theta}\,^{\,
m}\bullet\,R(x,y)\,\widehat{\theta},\nn\ee where $
R(x,y)\,\widehat{\theta}=\Big(\sum_{r=1}^{\infty}\big(\sum_{k=0}^{m-1}
\,{-Q_{k}(x)\over
Q_{m}(x)}\widehat{\beta}\bullet(\,\widehat{\theta}\,)^{\,\,
{m-k}}\,+{-f(x,y)\over
Q_{m}(x)}\,\widehat{\theta}\,\bullet(\,\widehat{\theta}\,)^{\,m}
\big)^{\,r}\Big)\bullet ({1\over Q_{m}(x)})\,\widehat{\beta}$. Thus,
\be {\rm
deg_{\widehat{\beta}}}[\Phi^{-1}\bullet\big(\widehat{\theta}^{\,
m}\bullet({1\over Q_{m}(x)}){\widehat{\beta}}\,\big)]= {\rm
deg_{\widehat{\beta}}}[\big(\widehat{\theta}^{\, m}\bullet({1\over
Q_{m}(x)}){\widehat{\beta}}\,\big)\bullet\Phi^{-1}]=
n-m,\label{esmat1}\nn\ee   \be {\rm
deg_{\widehat{\beta}}}[\Phi^{-1}\bullet\big(\,\widehat{\theta}\,^{\,
m}\bullet\,R(x,y)\,\widehat{\theta}\big)]< n-m,\hspace{1cm} {\rm
deg_{\widehat{\beta}}}[\big(\,\widehat{\theta}\,^{\,
m}\bullet\,R(x,y)\,\widehat{\theta}\big)\bullet \Phi^{-1}]<
n-m.\label{esmat2}\nn\ee From last two lines one can deduce, $
{\Deg}[\Psi^{-1}\bullet\Phi^{-1}]={\Deg}[\Phi^{-1}\bullet\Psi^{-1}]=n-m,$
thus, \be
{\Deg}[(\Psi\bullet\Phi)^{-1}]={\Deg}[(\Phi\bullet\Psi)^{-1}]=n-m>0.\nn\ee
Also, $ {\rm Coef}[\Psi^{-1}\bullet\Phi^{-1}]=
 {\rm Coef}[\Phi^{-1}\bullet\Psi^{-1}]={1\over Q_{m}(x)}{\rm
 Coef}[\Phi^{-1}]\in\PP,$ therefore,
 \be {\rm Coef}[(\Psi\bullet\Phi)^{-1}]=
 {\rm Coef}[(\Phi\bullet\Psi)^{-1}]={1\over Q_{m}(x)}{\rm Coef}[\Phi^{-1}]\in\PP. \nn\ee
It follows that $\Psi\bullet\Phi \in W_3$ and $\Phi\bullet\Psi \in
 W_3.$}

\item[]{6.\hspace{0.5cm}  For $\Psi$, $\Phi\in W_3$ we have $\Psi^{-1}$ and $\Phi^{-1}$ in
$W_2$, hence we can write, \be {\Deg} [\Psi^{-1}\bullet\Phi^{-1}]=
{\rm deg_{\widehat{\beta}}}[\Phi^{-1}\bullet\Psi^{-1}]= {\rm
deg_{\widehat{\beta}}}[\Psi^{-1}]+ {\Deg}[\Phi^{-1}]\geq 2, \nn\ee
\be {\Deg} [(\Psi\bullet\Phi)^{-1}]= {\rm
deg_{\widehat{\beta}}}[(\Phi\bullet\Psi)^{-1}]\geq 2,\nn \ee \be
{\rm Coef}[\Psi^{-1}\bullet\Phi^{-1}]=
 {\rm Coef}[\Phi^{-1}\bullet\Psi^{-1}]=  {\rm Coef}[\Phi^{-1}] {\rm Coef}[\Psi^{-1}]\in\PP,\nn\ee
 \be {\rm
Coef}[(\Psi\bullet\Phi)^{-1}]=
 {\rm Coef}[(\Phi\bullet\Psi)^{-1}]=  {\rm Coef}[\Phi^{-1}] {\rm Coef}[\Psi^{-1}]\in\PP.\nn\ee
Therefore in this case, $\Psi\bullet\Phi \in W_3$ and
$\Phi\bullet\Psi \in
 W_3. \nn$}\qed}
 \end{itemize}

We notice that elements in $W_3$ are in the form
$f(x,y)\widehat\theta$,
so we state the following lemma for these elements. First state the following condition.\\

{\bf Condition A1}, {We say $f(x,y)\in {\boldsymbol{B_2}}$ satisfies
condition $A1$ by order $n$ if
$\rd^{\,n-1}_{x}f(x,y)|_{\,x=y}=h(x)$, $h(x)\in\PP$, and (if $n\geq
2$) $\rd^{\,i}_{x}f(x,y)|_{\,x=y}=0$ for $i=0,1,...,n-2$.\\}

\begin{lem}\label{ghasem}{$f(x,y)\,\widehat\theta \in W_3$ with ${\Deg}
(f(x,y)\,\widehat\theta)^{-1}=n$ if and only if $f(x,y)$ satisfies
condition A1 by order $n$.\\}
\end{lem}
{\bf Proof.} {For $f(x,y)\,\widehat\theta\in W_3$ with ${\Deg}
(f(x,y)\,\widehat\theta)^{-1}=n$, by definition there exist
${\mathcal{O}}(x,\rd_{x})\,\widehat{\beta}=\big(\sum^{n}_{i=1}
P_{i}(x) \,\rd_{x}^{\,i}+P_{0}(x)\big)\,\widehat{\beta}$ (where
$P_{n}(x)\in \PP$) and $R(x,y)\widehat\theta$ such that, \bea
(f(x,y)\,\widehat\theta)&=&\Big(\big(\sum^{n}_{k=1} P_{k}(x)
\rd_{x}^{\,k}+P_{0}(x)\big)\,\widehat{\beta}+
 R(x,y)\,\widehat{\theta}\Big)^{-1}\nn\\&=&\widehat{\theta}^{\,
n}\bullet({1\over
P_{n}(x)})\,\widehat{\beta}+\,\widehat{\theta}\,^{\,
n}\bullet\,g(x,y)\,\widehat{\theta}.\label{khoda11}\nn\eea where $
g(x,y)\,\widehat{\theta}=\Big(\sum_{r=1}^{\infty}\big(\sum_{k=0}^{n-1}
\,{-P_{k}(x)\over
P_{n}(x)}\widehat{\beta}\bullet(\,\widehat{\theta}\,)^{\,\,
{n-k}}\,+{-R(x,y)\over
P_{n}(x)}\,\widehat{\theta}\,\bullet(\,\widehat{\theta}\,)^{\,n}
\big)^{\,r}\Big)\bullet ({1\over P_{n}(x)})\,\widehat{\beta}.$ From
last line we have,

\be  {\pi}_{\widehat\beta} [\rd_{x}^{\,n}{\,\widehat\beta} \bullet
f(x,y)\,\widehat\theta\,\,]={1\over
P_{n}(x)}{\,\widehat\beta}\hspace{0.2cm};\hspace{0.3cm}  {1\over
P_{n}(x)}\in \PP.\label{amini}\ee

The following relation can be proved by induction. \be
\rd^{\,n}_{x}{\,\widehat\beta}\bullet f(x,y)\,\widehat\theta\, =
\sum^{n-1}_{i=0} \sum^{i}_{k=0} {\tbinom ik} \Big(\rd^{\,k}_{x} \big
(\rd^{\,n-i-1}_{x}
f(x,y)|_{x=y}\big)\Big)\rd^{\,i-k}_{x}{\,\widehat\beta} +
(\rd^{\,n}_{x} f(x,y))\,\widehat\theta,\label{mehdi1}\ee

(In (\ref{mehdi1}) we take $\rd^{\,0}_{x}\equiv 1$). Inserting
$\rd^{\,n}_{x}{\,\widehat\beta}\bullet f(x,y)\,\widehat\theta\,$
from (\ref{mehdi1}) into equation (\ref{amini}) one can easily
conclude that $f(x,y)$ satisfies condition $A1$ by order $n$.\\

On the other hand let $f(x,y)\in \boldsymbol B_2$ and suppose it
satisfies condition A1 by order $n$. Thus by $n$ times applying
$\rd_{x}\,\widehat\beta$
 to $f(x,y)\,\widehat\theta$ we will have,
$\rd_{x}^{n}\,\widehat\beta\bullet
f(x,y)\,\widehat\theta=h(x)\,\widehat\beta+(\rd^{\,n}_{x}f(x,y))\,\widehat\theta
$. Hence \be f(x,y)\widehat\theta\,\, =\,\, (\,\widehat\theta\,)^{\,
n}\,\bullet h(x)\,\widehat\beta+ (\,\widehat\theta\,)^{\, n}\bullet
(\rd^{\,n}_{x}f(x,y))\,\widehat\theta \,\,=\,\,
(\,\widehat\theta\,)^{\, n}\,\bullet h(x)\,\widehat\beta\bullet
\Big(\,\widehat\beta \,+\, ({\rd^{\,n}_{x}f(x,y)\over
h(x)})\,\widehat\theta\,\Big). \nn\ee Therefore,
\bea(f(x,y)\,\widehat\theta\,)^{-\,
1}&=&\,(\,\widehat\beta+g(x,y)\widehat{\theta})\bullet {1\over
h(x)}\,\widehat\beta\bullet\rd_{x}^{n}\,\widehat\beta\nn\\&=&{1\over
h(x)}\rd_{x}^{n}\,\widehat\beta+(g(x,y)\widehat\theta)\bullet
{1\over
h(x)}\,\widehat\beta\bullet\rd_{x}^{n}\,\widehat\beta\label{moses},\eea
where $g(x,y)\widehat{\theta} =\sum^{\infty}_{i=1}
({-\rd^{\,n}_{x}f(x,y)\over h(x)}\,\widehat\theta\,)^{\,i}$. For the
second term in (\ref{moses}) by lemma (\ref{mahmood}) we have,\be
{\rm deg}_{\widehat\beta}[g(x,y)\widehat\theta\bullet {1\over
h(x)}\,\widehat\beta\bullet\rd_{x}^{n}\,\widehat\beta]<n.\label{moses2}\nn\ee
From (\ref{moses}) and last line we have
${\Deg}[(f(x,y)\,\widehat\theta\,)^{-\, 1}]=n, \hspace{0.4cm}
{\Cof}[(f(x,y)\,\widehat\theta\,)^{-\, 1}]={1\over h(x)}\in\PP$.
Thus by definition $(f(x,y)\,\widehat\theta\,)\in W_3$. }\qed\\

\begin{thm}\label{yarali}{ $\mathcal U$ is the whole group of units of the algebra $\mIs$}
\end{thm}

{\bf Proof.} Consider the following sets which give a partition of
$\mIs$:\bea{\mIs}_1=\{\Psi\in \mIs\,|\,\,\, {\rm
deg_{\widehat{\beta}}}[\Psi]=0;{\rm Coef[\Psi]} \not\equiv0 \},\nn\\
{\mIs}_2=\{\Psi\in \mIs\,|\,\,\, {\rm
deg_{\widehat{\beta}}}[\Psi]\ge1;{\rm Coef[\Psi]} \not\equiv0 \},
\nn\\ {\mIs}_3=\{\Psi\in \mIs\,|\,\,\, {\rm
deg_{\widehat{\beta}}}[\Psi]=0; {\rm Coef[\Psi]}=0 \}.\nn\eea
Clearly $ W_1 \subset\mIs_1, \hspace{0.1cm} W_2 \subset\mIs_2\, $
and $W_3 \subset\mIs_3.$ Put
$\boldsymbol{B{'}}=\boldsymbol{{B}_1}-\PP$ and consider the
following sets, \be {W'}_{1}=\mIs_{1}-W_1=\{\Psi\in \mIs\,|\,\,\,
{\rm deg_{\widehat{\beta}}}[\Psi]=0; {\rm Coef[\Psi]} \not\equiv0;
{\rm Coef[\Psi]}\in \boldsymbol{B{'}} \},\nn \ee
\be{W'}_{2}={\mIs_{2}}-W_2=\{\Psi\in \mIs\,|\,\,\, {\rm
deg_{\widehat{\beta}}}[\Psi]\ge1; {\rm Coef[\Psi]} \not\equiv0; {\rm
Coef[\Psi]}\in \boldsymbol{B{'}}\}, \nn\ee \be
{W'}_{3}={\mIs}_3-W_3.\nn\ee So in order to prove the theorem we
shall prove that no element of ${W'}_{1}\cup {W'}_{2} \cup
{W'}_{3}$ is invertible. \\

1-Assume $\Psi\in {W'}_{1}$ and consider the following cases. \\

a)\,\,Suppose  there is no subinterval $[c,d]\subset[a,b]$ on which
$\Cof[\Psi]$ is identically zero. Then its zero set has empty
interior; hence $\Cof[\Psi]P\not\equiv0$ for every
$P\in\boldsymbol B_1$ with $P\not\equiv0$. Thus by Lemma
\ref{salar2} the product of $\Psi$ with an element in $\mIs_1$ is in
$\mIs_1$; the product of $\Psi$ with an element of $\mIs_2$ is in
$\mIs_2$ and the product of $\Psi$ with an element of $\mIs_3$ is in
$\mIs_3$. Considering the fact that  the identity $\widehat{\beta}$
is in $\mIs_1$, if $\Phi$ happens to be the inverse of $\Psi$ we
should have $\Phi\in\mIs_1$ and therefore by Lemma \ref{salar2} \be
{\pi}_{\widehat{\beta}}[\Psi\bullet\Phi]=({\Cof}[\Psi\bullet\Phi])\widehat{\beta}=\widehat{\beta},
\hspace{1.5cm} {\Cof}[\Psi\bullet\Phi]=({\Cof}\Psi) ({\Cof}\Phi)=1.
\label{nima}\nn\ee But this is a contradiction since
$({\Cof}\Psi)\in \boldsymbol {B'}$.\\

b) Suppose there is a subinterval $[c,d]\subset[a,b]$ on which
$\Cof[\Psi]$ is zero. An inverse $\Phi$ cannot belong to $\mIs_3$,
since then ${\pi}_{\widehat\beta}[\Phi\bullet\Psi]=0$, nor to
$\mIs_1$, since its $\widehat\beta$-coefficient would vanish on
$[c,d]$. Hence assume
$\Psi=Q(x)\widehat{\beta}+g(x,y)\,\widehat{\theta}$ and
$\Phi=\big(\sum^{n}_{k=1} P_{k}(x)
\rd_{x}^{\,k}+P_{0}(x)\big)\,\widehat{\beta}+
 f(x,y)\,\widehat{\theta}\in\mIs_2$ with
$\Phi\bullet\Psi=\widehat\beta$. Using relation (\ref{mehdi1}), we should have, \bea &&\sum^{n}_{i=1}
\sum^{i}_{k=0} {\tbinom ik} \,P_{i}(x)\,\big(\rd^{\,k}_{x}
Q(x)\big)\,\rd^{\,i-k}_{x}\,+\,P_{0}(x)
\,Q(x)\nn\\&&+\sum^{n}_{n'=1}P_{n'}(x)\sum^{n'-1}_{i=0}
\sum^{i}_{k=0} {\tbinom ik} \Big(\rd^{\,k}_{x} \big
(\rd^{\,n'-i-1}_{x}
g(x,y)|_{x=y}\big)\Big)\rd^{\,i-k}_{x}\,\,=1.\label{zaki}\nn\eea

Comparing successively the coefficients of $\rd^{\,n}_{x}$,
$\rd^{\,n-1}_{x}$,...,$\rd_{x}$ gives, by descending induction,
$Q(x)P_{i}(x)\equiv0$ for $i=1,..,n$. Multiplying the above equation
by $Q(x)$ then gives $P_{0}(x)Q^{2}(x)=Q(x)$. Hence
$P_0=1/Q$ wherever $Q\ne0$; since $Q$ vanishes on a nonempty
interval but $Q\not\equiv0$, this contradicts the continuity of
$P_0$. Thus $\Psi$ is not invertible.\\

2-Let $\Psi\in {W'}_{2}$, then $\Psi=\big(\sum^{n}_{k=1} P_{k}(x)
\rd_{x}^{\,k}+P_{0}(x)\big)\,\widehat{\beta}+
 f(x,y)\,\widehat{\theta}$ where $P_{n}(x)\in \boldsymbol{B'}$. So

\be\Psi=\,\big( P_{n}(x)\,\widehat{\beta}-g(x,y)\,\widehat{\theta}
\big) \,\bullet
\,{\rd^{\,n}_{x}}\,\widehat{\beta},\label{nima2}\nn\ee where
$g(x,y)\,\widehat{\theta}= -\sum_{k=0}^{n-1}
\,{P_{k}(x)}\widehat{\beta}\bullet(\,\widehat{\theta}\,)^{\,
{n-k}}\,-{f(x,y)}\,\widehat{\theta}\,\,\bullet\,(\,\widehat{\theta}\,)^{n}$.
By part 1, $\big( P_{n}(x)\,\widehat{\beta}-g(x,y)\,\widehat{\theta}
\big)\in {W'}_{1}$, hence it is not invertible. Therefore $\Psi$ can
not be invertible, being a product of invertible element
$\,{\rd^{\,n}_{x}}\,\widehat{\beta}$ with the non-invertible element
$\big( P_{n}(x)\,\widehat{\beta}-g(x,y)\,\widehat{\theta} \big)$.
 \\

3-Let $\Psi=f(x,y)\widehat{\theta}\in {W'}_{3}$. By Lemma
(\ref{ghasem}) there can be only two cases
for $f(x,y)\in \boldsymbol{{B}_2}$ as following.\\

a) There exists $n$ such that
$\rd^{\,n-1}_{x}f(x,y)|_{\,x=y}=h(x)\in \boldsymbol{B'}$, and (if
$n\geq 2$) $\rd^{\,i}_{x}f(x,y)|_{\,x=y}=0$ for $i=0,1,...,n-2$. We
then have $(\rd_{x}^{\,n}\,\widehat{\beta})\bullet
f(x,y)\,\widehat{\theta}=(h(x)\,\widehat{\beta}-g(x,y)\,\widehat{\theta})$,
thus $\Psi=
(\,\widehat{\theta}\,)^{n}\bullet\big(h(x)\,\widehat{\beta}-g(x,y)\,\widehat{\theta}\big).$
Therefore $\Psi$ cannot be
invertible.\\

b) $f(x,x)=0$\, and\, $\rd^{\,n}_{x}f(x,y)|_{\,x=y}=0$ for all $\,n \in \mathbb{N}$.\\

In this case we have, $\forall \,k\in
\mathbb{N},\hspace{0.25cm}(\rd_{x}^{\,k}\,\widehat{\beta})\bullet
f(x,y)\,\widehat{\theta}=\big(\rd^{k}_{x}f(x,y)\big)\,\widehat{\theta}$,
and thus for any general element $\Phi=\big((\sum^{n}_{k=1} P_{k}(x)
\rd_{x}^{\,k}+P_{0}(x))\,\widehat{\beta}+
 R(x,y)\,\widehat{\theta}\,\big)\in \mIs$ we have,

\be\Phi\bullet\Psi=\big((\sum^{n}_{k=1} P_{k}(x)
\rd_{x}^{\,k}+P_{0}(x))\,\widehat{\beta}+
 R(x,y)\,\widehat{\theta}\,\big)\bullet f(x,y)\,\widehat{\theta}=g(x,y)\,\widehat{\theta},\label{nima4}\nn\ee where $g(x,y)=\big(\sum^{n}_{k=1} P_{k}(x) \rd_{x}^{\,k}
f(x,y)+P_{0}(x) f(x,y)\big)+
 \int^{x}_{y} dz\,R(x,z) f(z,y)$. From above we obviously
 have, \be {{\pi}_{\widehat{\beta}}}[\Phi\bullet
f(x,y)\,\widehat{\theta}]=0.\nn\ee Thus, $\nexists \Phi\in\mIs$ with
$\Phi\bullet f(x,y)\,\widehat{\theta}=\widehat{\beta}.$ Therefor the
element $\Psi=f(x,y)\,\widehat{\theta}$ is not
invertible in $\mIs$. \qed\\

We finish this section with the following corollaries which have
application in the solution of Volterra integral equation of first
kind.

\begin{cor}\label{koloshani2}{Consider $k(x,y)\,\widehat{\theta}$ and ${\mathcal{Q}_{\,n}}(x,\rd_{x})\,\widehat{\beta}$ with
${\Deg}[{\mathcal{Q}_{\,n}}(x,\rd_{x})\,\widehat{\beta}]=n$. If \be
\Big(\sum_{k=1}^{\infty} (k(x,y)\widehat{\theta})^{\, k}\Big)\bullet
{\mathcal{Q}_{\,n}}(x,\rd_{x})\,\widehat{\beta}=\mathcal{O}(x,\rd_{x})\,\widehat{\beta}+
R(x,y)\,\widehat{\theta},\nn\ee\\
then for $\displaystyle\Phi=\sum_{k=1}^{n}
\big((k(x,y)\widehat{\theta})^{\, k} \bullet
{{\mathcal{Q}_{\,n}}}(x,\rd_{x})\,\widehat{\beta}\big)$ we have $
\mathcal{O}(x,\rd_{x}) \,\widehat{\beta} =
{\pi}_{\,\widehat{\beta}}[\Phi].$\\
Furthermore assuming $\displaystyle
C(x,y)\,\widehat{\theta}={\pi}_{\,\widehat{\theta}}[\Phi]$ and \bea
S(x)\,\widehat{\beta}+D(x,y)\,\widehat{\theta}=(k(x,y)\widehat{\theta})^{\,
n}\bullet
{\mathcal{Q}_{\,n}}(x,\rd_{x})\,\widehat{\beta},\hspace{1cm}
H(x,y)\,\widehat{\theta}=\big(\sum_{r=1}^{\infty}
(k(x,y)\widehat{\theta})^{\, r}\big), \nn\eea we have: \be
R(x,y)\widehat{\theta} =
C(x,y)\widehat{\theta}+(S(y)H(x,y))\,\widehat{\theta}+H(x,y)\,\widehat{\theta}\bullet
D(x,y)\,\widehat{\theta}.\nn\ee\\}
\end{cor}
{\bf Proof.} From Lemma \ref{mahmood} we have,
$(k(x,y)\widehat{\theta})^{n}\bullet
{\mathcal{Q}_{n}}(x,\rd_{x})\,\widehat{\beta}=
S(x)\,\widehat{\beta}+D(x,y)\,\widehat{\theta}$. Now

\bea
 \big(\sum_{k=1}^{\infty} (k(x,y)\widehat{\theta})^{\,
k}\big )&\bullet
&{\mathcal{Q}_{n}}(x,\rd_{x})\,\widehat{\beta}=\sum_{k=1}^{\infty}
\big((k(x,y)\widehat{\theta})^{\, k}
 \bullet {{\mathcal{Q}_{n}}}(x,\rd_{x})\,\widehat{\beta}\,\,\big)\nn\\&=& \sum_{k=1}^{n}
\big((k(x,y)\widehat{\theta})^{\, k} \bullet
{{\mathcal{Q}_{n}}}(x,\rd_{x})\,\widehat{\beta}\,\big)+
\sum_{r=1}^{\infty} (k(x,y)\widehat{\theta})^{\, r}
 \bullet \big((k(x,y)\widehat{\theta})^{n}\bullet
{\mathcal{Q}_{n}}(x,\rd_{x})\,\widehat{\beta}\,\big) \nn\\&=&
\sum_{k=1}^{n} \big((k(x,y)\widehat{\theta})^{\, k} \bullet
{{\mathcal{Q}_{n}}}(x,\rd_{x})\,\widehat{\beta}\,\big)+
\sum_{r=1}^{\infty} (k(x,y)\widehat{\theta})^{\, r}
 \bullet (S(x)\,\widehat{\beta}+D(x,y)\,\widehat{\theta})\nn\\&=&
{\pi}_{\,\widehat{\beta}}[\Phi]+{\pi}_{\,\widehat{\theta}}[\Phi]+(S(y)H(x,y))\,\widehat{\theta}+H(x,y)\,\widehat{\theta}\bullet
D(x,y)\,\widehat{\theta}. \nn \hspace{4cm}\qed\eea

\begin{cor}\label{koloshani3}{Let $f(x,y)\,\widehat{\theta}\in W_3$
and ${\Deg}[(f(x,y)\,\widehat{\theta})^{-1}]=n$. Suppose \be
(f(x,y)\widehat\theta\,\,)^{-1}=\mathcal{O}(x,\rd_{x})\,\widehat{\beta}+
R(x,y)\widehat{\theta}. \nn\ee Then by taking
$h(x)=\rd^{\,n-1}_{x}f(x,y)|_{x=y}$ and, \bea \sum^{n}_{i=1}\big(
({-\rd^{\,n}_{x}f(x,y)\over h(x)}\,\widehat\theta\,)^{\, i} \bullet
{1\over
h(x)}\rd_{x}^{\,n}\,\widehat\beta\big)&=&\Phi,\hspace{2.75cm}
{\pi}_{\,\widehat{\theta}}[\Phi]=C(x,y)\,\widehat{\theta}, \nn\\
({-\rd^{\,n}_{x}f(x,y)\over h(x)}\,\widehat\theta\,)^{\, n} \bullet
{1\over h(x)}\rd_{x}^{\,n}\,\widehat\beta &=&
S(x)\,\widehat\beta+D(x,y)\widehat{\theta},\hspace{1cm}
\sum_{k=1}^{\infty} ({-\rd^{\,n}_{x}f(x,y)\over
h(x)}\widehat{\theta})^{\,k}=H(x,y)\,\widehat{\theta}, \nn\eea we
have: \bea \mathcal{O}(x,\rd_{x}) \,\widehat{\beta} &=&
{1\over h(x)}\rd_{x}^{\,n}\,\widehat\beta+{\pi}_{\,\widehat{\beta}}[\Phi], \nn\\
R(x,y)\,\widehat{\theta} &=&
C(x,y)\,\widehat{\theta}+(S(y)H(x,y))\,\widehat{\theta}+H(x,y)\,\widehat{\theta}\bullet
D(x,y)\,\widehat{\theta}.\hspace{1cm}\label{lol}\nn\eea\\}
\end{cor}
{\bf Proof.} From equation (\ref{moses}) we have,
\be(f(x,y)\widehat\theta\,\,)^{-\, 1}={1\over
h(x)}\rd_{x}^{\,n}\,\widehat\beta\,\,+\,\,(\sum^{\infty}_{i=1}
({-\rd^{\,n}_{x}f(x,y)\over h(x)}\,\widehat\theta\,)^{ i})\bullet
{1\over
h(x)}\,\widehat\beta\bullet\rd_{x}^{\,n}\,\widehat\beta.\nn\ee Now
the result is proved by applying Corollary \ref{koloshani2} to the
second term. \qed\\

\section {Applications}
\label{app}

The results of previous section concerns the space $M[a,b]$. However
by the following result we may extend the application of the unit
group $\mathcal U$ to endomorphism of $C[a,b]$ and to operators
between ${D}^{n}[a,b]$ and ${C}[a,b]$ where ${D}^{n}[a,b]$ is the
space of $n$ times continuously differentiable functions
$\phi(x):[a,b]\to\mc$ such that ${\rd^{k}_{x}}\phi(x)|_{x=a}=0$ for
$k=0\cdots,n-1$.

\begin{pro}\label{Haghany}{Let $\bf{j}$ and $\bf{k}$ be as in Lemma
\ref{iraj45}. Then \\
1- $\bf{j}$ can be viewed as an invertible endomorphism of $C[a,b]$\\
2- $\bf{k}$ can be viewed as an invertible operator from
${D}^{n}[a,b]$ to ${C}[a,b]$.}\end{pro}

{\bf Proof.} For 1 see \cite{it}, section 1.2. \\
For the proof of 2 consider ${\bf{k}}=\sum^{n}_{i=1} P_{i}(x)
\,\rd_{x}^{\,i}+P_{0}(x)+
f(x,y)\theta(x-y):{D}^{n}[a,b]\to{C}[a,b]$. If
$\phi(x)\in{D}^{n}[a,b]$, then using Cauchy formula for repeated
integration we have for $1 \le r\le n$,
\be\rd^{\,n-r}_{x}\phi(x)=\int^{x}_{a}dz_1\int^{z_1}_{a}dz_2\cdots\int^{z_{r-1}}_{a}dz_{r}\,\,(\rd^{\,n}_{z_{r}}\phi(z_r))=\int^{x}_{a}dz\,
{(x-z)^{r-1}\over (r-1)!}\rd^{\,n}_{z}\phi(z).\nn\ee Using the above
we thus find \be {\bf{k}}(\phi(x))= P_{n}(x)\,\rd^{n}_{x}\phi(x)+
\int^{x}_{a} dz h(x,z)\rd^{n}_{z}\phi(z), \nn\ee
 where $h(x,y)=\sum^{n-1}_{i=0} {P_{i}(x)}
\,{(x-y)^{n-i-1}\over {(n-i-1)!}}+ \int^{x}_{y}
dz\,{f(x,z)}{(z-y)^{n-1}\over {(n-1)!}}.$
Therefore $\bf{k}$ is the composition of two operators,
$\bf{k}=\bf{r}*\bf{d}$ where \bea  &&{\bf{d}}:{D}^{n}[a,b] \to
C[a,b], \hspace{0.5cm} {\bf{d}}\,(\phi(x))=\rd^{n}_{x}\phi(x),\nn\\
&&{\bf{r}}:{C}[a,b] \to C[a,b], \hspace{0.8cm}
{\bf{r}}\,(\phi(x))=P_{n}(x)\phi(x)+\int^{x}_{a}dz h(x,z)\phi(z).
\nn\eea Both $\bf{d}$ and $\bf{r}$ are invertible:
\be{\bf{d}}^{-1}(\phi(x))=\int^{x}_{a} dz
{(x-z)^{n-1}\over(n-1)!}\phi(z),\hspace{0.75cm}
{\bf{r}}^{-1}(\phi(x))={(1/P_{n}(x))\phi(x)}+\int^{x}_{a} dz
R(x,z)\phi(z).\nn\hspace{1cm}\qed\ee\\

In the rest of this section, we shall make use of the information
obtained in previous section on invertible operators in finding the
solution of differential and integral equations. First we consider
the equation: \be P(x)\,y_1(x)+\int^{x}_{a} dz f(x,z)\,
y_1(z)=g(x),\label{voltra}\ee where $f(x,y)\in {\boldsymbol B_2}$,
$P(x)\in {\boldsymbol \PP}$ and $g(x)\in C[a,b]$. From Proposition
\ref{Haghany} we have, \be
y_{1}(x)=(P(x)+f(x,y)\theta(x-y))^{-\underline 1}\cdot\, g(x),\nn\ee
where \bea(P(x)+f(x,y)\theta(x-y))^{-\underline
1}\equiv({P(x)\widehat{\beta}+f(x,y)\widehat{\theta}})^{-1}&=&({1\over
P(x)})\,\widehat{\beta}+\big(\sum^{\infty}_{n=1} \,\,({-f(x,y)\over
P(x)}\,\widehat{\theta}\,)^{\, n}\big)\bullet({1\over
P(x)})\,\widehat{\beta}\nn\\&\equiv&({1\over P(x)})+\Big(({1\over
P(y)})R(x,y)\Big)\theta(x-y).\label{khaleghi2}\nn\eea Next we
consider \bea \sum^{n}_{k=0} P_{k}(x) {\rd^{\,k}_{x}} \,y_2(x)
+\int^{\,x}_{a} dz
\,f(x,z)\,y_2(z)=\,g(x),\label{ghar}\\V^{B.C.}_{causal} \,\,= \,\,\{
\rd_{x}^{\,n} y_{2}(a)=0\,|n=0,1,\cdots n-1\}\label{zahir},\eea
where $P_{i}(x)\in {\boldsymbol B_1}$ for $i=0,1,\cdots,n$;
$P_{n}(x)\in\PP$ and $f(x,y)\in {\boldsymbol B_2}$ subject to causal
boundary condition ($V^{B.C.}_{causal}$). Again from Proposition
\ref{Haghany} we have, \be y_{2}(x)=(\sum^{n}_{k=0} P_{k}(x)
{\rd^{\,k}_{x}} \,+\int^{\,x}_{a} dz \,f(x,z)\,)^{-\underline
1}\,\cdot\, g(x),\nn\ee where \bea &&(\sum^{n}_{k=0} P_{k}(x)
{\rd^{\,k}_{x}}+\int^{\,x}_{a} dz \,f(x,z)\,)^{-\underline
1}\,\,\equiv\,\,\Big(\big(\sum^{n}_{i=1} P_{i}(x)
\rd_{x}^{\,i}+P_{0}(x)\big)\,\widehat{\beta}+
 f(x,y)\widehat{\theta}\,\Big)^{-1}\nn\\ &=&(\widehat{\theta}\,)^{\, n}\bullet({1\over
P_{n}(x)})\,\widehat{\beta}+(\widehat{\theta}\,)^{\, n}\bullet
\Big[\sum_{r=1}^{\infty}\big(\sum_{k=0}^{n-1} \,{-P_{k}(x)\over
P_{n}(x)}{\widehat{\beta}}\bullet(\widehat{\theta}\,)^{\,
{n-k}}\,+{-f(x,y)\over
P_{n}(x)}\widehat{\theta}\,\bullet\,(\widehat{\theta}\,)^{\, n}
\big)^{ r}\Big]\bullet ({1\over
P_{n}(x)})\,\widehat{\beta}\nn\\&\equiv&{1\over{P_{n}(y)}}
\Big({(x-y)^{n-1}\over(n-1)!}+\int^{x}_{y}dz\,{(x-z)^{n-1}\over(n-1)!}\,R(z,y)\Big)\theta(x-y)=T(x,y)\theta(x-y).\label{hamed}\eea
In the above, $R(x,y)$ is resolvent kernel of $\,
h(x,y)=-\sum_{k=0}^{n-1}{P_{k}(x)\over
P_{n}(x)}{{(x-y)^{n-k-1}}\over(n-k-1)!}-\int^{x}_{y}
dz\,{f(x,z)\over P_{n}(x)}{(z-y)^{n-1}\over(n-1)!}$. Thus we have,
\be y_{2}(x)=\int^{b}_{a} dz\, T(x,z)\theta(x-z)
g(z).\label{pari}\ee The function $G(x,y)=T(x,y)\theta(x-y)$ is
called {\em causal Green function} (e.g. see\cite{sta}). Working in
$\mIs$ we show $T(x,y)$ that was derived above satisfies two
conditions:\\\\
(1)-$\,\,\,\big(\sum^{n}_{k=0} P_{k}(x) {\rd^{k}_{x}}
\big)\,T(x,y)+\int^{x}_{y}\, dz f(x,z)T(z,y)=0,$ \\\\
(2)-$\,\,\,\rd^{\,i}_{x}( T(x,y))|_{x=y}=0$ for $i=0,1,..,n-2$ and
$\rd^{\,n-1}_{x}( T(x,y))|_{x=y}=1/P_{n}(x)$.\\

$T(x,y)\widehat{\theta}$ is in $W_3$, thus by Lemma \ref{ghasem},
$T(x,y)$ satisfies condition $A{1}$ by order $n$ which is equivalent
to condition (2) above. Second, we notice by taking
$\Psi=\big(\sum^{n}_{k=1} P_{k}(x)
\rd_{x}^{\,k}+P_{0}(x)\big){\,\widehat{\beta}}+f(x,y){\,\widehat{\theta}}$,
\be\Psi\bullet
T(x,y)\widehat{\theta}={\,\widehat{\beta}}\Longrightarrow
{\pi}_{\widehat{\theta}}\big[\Psi\bullet
T(x,y)\widehat{\theta}\,\,\big]=0.\nn \ee But since
${\pi}_{\widehat{\theta}}\big[\Psi\bullet
T(x,y)\widehat{\theta}\big]=\big(\sum^{n}_{k=1} P_{k}(x)
\rd_{x}^{\,k}T(x,y)+P_{0}(x)T(x,y)+\int^{x}_{y}dz
f(x,z)T(z,y)\Big)\widehat{\theta}$, the condition (1) holds. The
answer (\ref{hamed}) for $T(x,y)$ is given in \cite{ad} when
$f(x,y)\equiv 0$ and $P_{n}(x)=1$; also see \cite{polyan} section
9.2-3.\\

As our last case of application let us consider a Volterra integral
equation of first kind with a smooth kernel of {\em degree $n$},
defined by: \be\int^{x}_{a}\, dz\,f(x,z)
y(z)=\phi(x)\label{nsmooth},\ee where $f(x,y)\in {\boldsymbol B_2}$,
$\phi(x)\in D^{n}[a,b]$, $\rd^{\,n-1}_{x}f(x,y)|_{\,x=y}=h(x)\in\PP$
and (if $n\geq 2$) $\rd^{\,i}_{x}f(x,y)|_{\,x=y}=0$ for
$i=0,1,...,n-2$.

Using Proposition \ref{Haghany}  and Corollary \ref{koloshani3} we
get the answer in the form, \be y(x)=\mathcal{O}(x,\rd_{x})\phi(x)+
 \int^{x}_{a} dz \,R(x,z) \phi(z).\label{ell}\ee By Corollary \ref{koloshani3} one can derive
$\mathcal{O}(x,\rd_{x})$ explicitly and $R(x,y)$ by a series of
integrals for any degree $n\in \Na$ from the kernel $f(x,y)$. The
following Lemmas \ref{ferdosi1} and \ref{ferdosi2} may be derived by
using Corollary \ref{koloshani3} to give solution of Volterra
integral equation of first kind of degree $n=1,2$.
\begin{lem}\label{ferdosi1}{The solution of equation (\ref{nsmooth}) of order $n=1$
is given by (\ref{ell}) where \be
\mathcal{O}(x,\rd_{x})=p(x)\rd_{x}+q(x)\,;\hspace{1.5cm}
p(x)={1\over f(x,x)}\,;\,\, \hspace{1.5cm}
q(x)=-{(\rd_{x}{f(x,y)})|_{x=y} \over f(x,x)^2}.\nn\ee
$\displaystyle R(x,y)=C(x,y)+H(x,y) S(y)+\int^{x}_{y} dz
H(x,z)D(z,y)$ \hspace{0.1cm} in which,\be
C(x,y)=D(x,y)={{\rd_{y}\rd_{x} f(x,y)}\over{f(x,x)\,
f(y,y)}}-{{\rd_{x} f(x,y)\rd_{y} f(y,y)}\over{f(y,y)}^2},
\hspace{1.5cm} S(x)=-{\rd_{x} f(x,y)|_{x=y}\over f(x,x)^2},
\label{frds5}\nn\ee and $H(x,y)$ being the resolvent kernel of
$\displaystyle{-\rd_{x} f(x,y)\over f(x,x)}$, it is given by $
H(x,y)\widehat{\theta}=\sum^{\infty}_{n=1}\Big({-\rd_{x} f(x,y)\over
f(x,x)}\,\widehat{\theta}\,\Big)^{n}.$ \qed}
\end{lem}

\begin{lem}\label{ferdosi2}{The solution of equation (\ref{nsmooth}) of order $n=2$
is given by (\ref{ell}) where \be
\mathcal{O}(x,\rd_{x})=r(x)\,\rd_{x}^2+p(x)\rd_{x}+q(x),\label{expdofi1}\nn\ee
\be r(x)={1\over {(\rd_{x}{f(x,y)})|_{x=y}}}\,\,, \hspace{1cm}
p(x)=-{(\rd_{x}^2{f(x,y)})|_{x=y} \over
((\rd_{x}{f(x,y)})|_{x=y})^2}. \label{expdofi2}\nn\ee \\
Take  $g(x,y), h(x)$, $\lambda(x)$ and $S(x)$ as below, \be\,\,
g(x,y)=-{(\rd_{x}^2{f(x,y)}) \over (\rd_{x}{f(x,y)})|_{x=y}},
\hspace{0.3cm} h(x)=\rd_{x}f(x,y)|_{x=y},\hspace{0.3cm}
\lambda(x)=(\rd_{y}({g(x,y)\over h(y)}))|_{x=y},\hspace{0.3cm}
S(x)={(g(x,x))^2\over h(x)}\,. \label{ex2}\nn\ee
 \be\,\,
\label{ex22}\nn\ee The function  $q(x)$ is then given by
$\hspace{0.3cm}\displaystyle q(x)=-\rd_{y}({g(x,y)\over
h(y)})|_{x=y}+{{(g(x,x))}^{2}\over h(x)}$, while \be
R(x,y)=C(x,y)+H(x,y) S(y)+\int^{x}_{y} dz H(x,z) D(z,y),
\label{adel4}\nn\ee where, \bea C(x,y)=(\rd^{\,2}_{y}({g(x,y)\over
h(y)}))-(\rd_{y}({g(y,y)\over h(y)}))g(x,y)-({g(y,y)\over
h(y)})(\rd_{y}\,g(x,y))\nn\\-\lambda(y) g(x,y)+\int^{x}_{y}
dz\,g(x,z)(\rd^{2}_{y} ({g(z,y)\over h(y)})), \label{adel1}\nn\eea

\be D(x,y)=-(\rd_{y}({g(y,y)\over h(y)}))g(x,y)-({g(y,y)\over
h(y)})(\rd_{y}\,g(x,y))-\lambda(y) g(x,y)+\int^{x}_{y}
dz\,g(x,z)(\rd^{2}_{y} ({g(z,y)\over h(y)})),\label{adel2}\nn\ee and
$H(x,y)$ being resolvent kernel of $g(x,y)$, it is given by
$H(x,y)\widehat{\theta}=\sum^{\infty}_{n=1}(g(x,y)\,\widehat{\theta})^{\,n}.$
\qed\\}
\end{lem}
Finally we examine cases $n=1,2$ by using the following examples
from \cite{polyan}.\\

{\bf Example 1.} Consider integral equation
$\displaystyle\int^{x}_{a}[g(x)-g(z)+c]y(z)
dz=\phi(x),$ with $\phi(a)=0$ and $c\ne0$.\\

Put $f(x,y)=g(x)-g(y)+c$. According to Lemma (\ref{ferdosi1}) it is
easy to see,
 \be
\mathcal{O}(x,\rd_{x})=p(x)\rd_{x}+q(x),\hspace{1cm} p(x)={1\over
c}\,\,, \hspace{1cm} q(x)=-{1\over c^2}\rd_{x}g(x).\nn\ee Also we
see $D(x,y)=C(x,y)=0$, and $S(x)={{-\rd_{x}g(x)}\over c^2}$. Using
relation $(g(x)\,\widehat{\theta})^{\,n}={{g(x)\,(\int^{x}_{y}\,
dz\, g(z))}^{n-1} \over (n-1)!}\,\widehat{\theta}\hspace{0.5cm}$ one
can derive $H(x,y)$ as, \bea
H(x,y)\widehat{\theta}&=&\sum^{\infty}_{n=1}\Big({-\rd_{x}
f(x,y)\over
f(x,x)}\,\widehat{\theta}\,\Big)^{n}=\sum^{\infty}_{n=1}({-{\rd_{x}g(x)}\over
c}\,\widehat{\theta}\,)^{n}=({-{\rd_{x}g(x)}\over
c})\sum^{\infty}_{n=1}{{(\int^{x}_{y}\, dz\,{-{\rd_{z}g(z)}\over
c})}^{n-1} \over (n-1)!}\,\widehat{\theta}\nn\\
&=&({-{\rd_{x}g(x)}\over c} )\sum^{\infty}_{n=1}{({g(y)-g(x)\over
c})^{n-1} \over (n-1)!}\,\widehat{\theta}=({-{\rd_{x}g(x)}\over
c})\exp({g(y)-g(x)\over c})\,\widehat{\theta}.\nn\eea Finally since
$D(x,y)=C(x,y)=0$, we have $
R(x,y)=S(y)H(x,y)=({{\rd_{y}g(y)}{\rd_{x}g(x)}\over
c^3})\exp({g(y)-g(x)\over c}).$ It follows that the solution,\bea
y(x)&=&{1\over c}\rd_{x}\phi(x)-{1\over c^2}(\rd_{x}g(x)) \phi(x)
+\int^{x}_{a} dz \,({\rd_{x}g(x)\,\rd_{z}g(z)\over
c^3})\exp[{{g(z)-g(x)}\over c}]\phi(z)\nn\\&=&{1\over
c}\rd_{x}\phi(x)-{1\over c^2}\rd_{x}g(x)\int^{x}_{a}dz
\exp[{g(z)-g(x)\over c}]\rd_{z}\phi(z),\nn\eea is in accordance with
\cite{polyan} section 1.9-1 example 3.\\

{\bf Example 2.} Consider integral equation
$\displaystyle\int^{x}_{a}[e^{\alpha(x-z)}-1]\,y(z) dz=\phi(x),\,\,$
where $\alpha\ne0$, with
$\phi(a)=\rd_{x}\phi(x)|_{x=a}=0$.\\

Put $f(x,y)=e^{\alpha(x-y)}-1$. By Lemma (\ref{ferdosi2}) we get,
\be \, g(x,y)=-\alpha\, e^{\,\alpha(x-y)}, \hspace{0.5cm}
h(x)=\alpha, \hspace{0.5cm} \lambda(x)=\alpha\,,\hspace{0.5cm}
S(x)=\alpha.\nn\ee The differential part is then given by $
\mathcal{O}(x,\rd_{x})=r(x)\rd^2_{x}+p(x)\rd_{x}+q(x),$ where \be
r(x)={1\over h(x)}={1\over \alpha}\,\,,\hspace{0.5cm}
p(x)={g(x,x)\over h(x)}=-1\,,\hspace{0.5cm}
q(x)=-\rd_{y}({g(x,y)\over h(y)})|_{x=y}+{{(g(x,x))}^{2}\over
h(x)}=-\alpha+\alpha=0.\nn\ee For the integral part we find, \be
C(x,y)={\alpha}^3\,e^{\alpha(x-y)}(x-y)+{\alpha}^2\,e^{\alpha(x-y)},\hspace{2cm}D(x,y)={\alpha}^3\,e^{\alpha(x-y)}(x-y)+2\,{\alpha}^2\,e^{\alpha(x-y)}.
\nn\ee Now using relation
$(e^{\alpha(x-y)}\,\widehat{\theta})^{\,n}=({e^{\alpha(x-y)}{\,(x-y)}^{n-1}
\over (n-1)!})\,\widehat{\theta}$, \bea
H(x,y)\,\widehat{\theta}&=&\sum^{\infty}_{n=1}(g(x,y)\,\widehat{\theta})^{n}=\sum^{\infty}_{n=1}(-\alpha\,
e^{\,\alpha(x-y)}\,\widehat{\theta})^{n}=\sum^{\infty}_{n=1}({e^{\alpha(x-y)}(-\alpha)^{n}{\,(x-y)}^{n-1}
\over (n-1)!})\,\widehat{\theta}\nn\\
&=&\Big(-\alpha\,e^{\alpha(x-y)}\big(\sum^{\infty}_{m=0}{(\alpha{(y-x)})^{m}
\over
m!}\big)\Big)\,\widehat{\theta}=(-\alpha\,e^{\alpha(x-y)}e^{\alpha(y-x)})
\,\widehat{\theta}=(-\alpha)\,\widehat{\theta}.\nn\eea So
$H(x,y)=-\alpha$. We calculate $R(x,y)$ as, \be
R(x,y)={\alpha}^3\,e^{\alpha(x-y)}(x-y)+{\alpha}^2\,e^{\alpha(x-y)}-{\alpha}^2+\int^{x}_{y}
dz
(-\alpha)({\alpha}^3\,e^{\alpha(z-y)}(z-y)+2\,{\alpha}^2\,e^{\alpha(z-y)})=0.\nn\hspace{0.75cm}\ee
Finally we get  the solution $ y(x)={1\over
\alpha}\rd^{\,2}_{x}\phi(x)-\rd_{x}\phi(x)$ in accordance with
\cite{polyan} section 1.2-1 example 3.



\end{document}